\documentclass[11pt, reqno]{amsart}
\usepackage{color}
\usepackage[utf8]{inputenc}
\usepackage[T1]{fontenc}

\usepackage{amsmath}
\usepackage{amssymb}
\usepackage{amsfonts}
\usepackage{mathrsfs}
\usepackage{amsbsy}
\usepackage{relsize}
\usepackage[colorlinks, linkcolor=red, citecolor=blue, urlcolor=blue, pagebackref, hypertexnames=false]{hyperref}
\usepackage[hyperpageref]{backref}
\usepackage{enumitem}

\usepackage{tikz}

\numberwithin{equation}{section}
\newcounter{bbb}
\numberwithin{bbb}{section}

\newcommand{\R}{\mathbb{R}}

\newcommand{\C}{\mathbb{C}}

\newcommand{\what}{\widehat}
\newcommand{\bD}{\mathbf{D}}

\newcommand{\beq}{\begin{eqnarray}}
\newcommand{\eeq}{\end{eqnarray}}
\newcommand{\bq}{\begin{equation}}
\newcommand{\eq}{\end{equation}}
\newcommand{\beqn}{\begin{eqnarray*}}
\newcommand{\eeqn}{\end{eqnarray*}}
\newcommand{\bex}{\begin{exo}}
\newcommand{\eex}{\end{exo}}
\newcommand{\ben}{\begin{enumerate}}
\newcommand{\een}{\end{enumerate}}

\newtheorem{thm}[bbb]{Theorem}
\newtheorem{lem}[bbb]{Lemma}
\newtheorem{prop}[bbb]{Proposition}
\newtheorem{cor}[bbb]{Corollary}
\newtheorem{defi}[bbb]{Definition}

\newtheorem{rem}[bbb]{Remark}

\title[Focusing Inhomegeneous Fractional NLS]{Long time dynamics for the focusing inhomogeneous fractional  Schr\"odinger equation}

\author[M. Majdoub and T. Saanouni]{Mohamed Majdoub and Tarek Saanouni}
\address[M. Majdoub]{Department of Mathematics, College of Science, Imam Abdulrahman Bin Faisal University, P. O. Box 1982, Dammam, Saudi Arabia.\newline Basic and Applied Scientific Research Center, Imam Abdulrahman Bin Faisal University, P.O. Box 1982, 31441, Dammam, Saudi Arabia}
\email{\sl mmajdoub@iau.edu.sa}
\email{\sl med.majdoub@gmail.com}
\address[T. Saanouni]{ Departement of Mathematics, College of Science and Arts in Uglat Asugour, Qassim University, Buraydah, Kingdom of Saudi Arabia.\newline University of Tunis El Manar, Faculty of Sciences of Tunis, LR03ES04 partial differential equations and applications, 2092 Tunis, Tunisia.}
\email{\sl Tarek.saanouni@ipeiem.rnu.tn}
\email{\sl t.saanouni@qu.edu.sa}

\subjclass[2020]{35Q55, 35P25, 35R11, 35B44, 47J35.}
\keywords{Fractional Schr\"odinger equation, global well-posedness, scattering, Morawetz estimates, Virial identities, blow-up.}

\begin{document}
\begin{abstract}
We consider the following fractional NLS with focusing inhomogeneous power-type nonlinearity
$$i\partial_t u -(-\Delta)^s u  + |x|^{-b}|u|^{p-1}u=0,\quad (t,x)\in \R\times \R^N,$$
where $N\geq 2$, $1/2<s<1$, $0<b<2s$ and $1+\frac{2(2s-b)}{N}<p<1+\frac{2(2s-b)}{N-2s}$.

We prove the ground state threshold of global existence and scattering versus finite time blow-up of energy solutions in the inter-critical regime with spherically symmetric initial data. The scattering is proved by the new approach of Dodson-Murphy ({Proc. Am. Math. Soc.} {145}: {4859--4867}, 2017). This method is based on Tao's scattering criteria and Morawetz estimates. One describes the threshold using some non-conserved quantities in the spirit of the recent paper by Dinh (Discr. Cont. Dyn. Syst. 40: 6441--6471, 2020). The radial assumption avoids a loss of regularity in Strichartz estimates. The challenge here is to overcome two main difficulties. The first one is the presence of the non-local fractional Laplacian operator. The second one is the presence of a singular weight in the non-linearity. The greater part of this paper is devoted to prove the scattering of global solutions in $H^s(\R^N)$.                       \end{abstract}
\date{\today}
	\maketitle
\tableofcontents

\section{Introduction}
\label{S1}
	We consider the following fractional NLS with focusing inhomogeneous power nonlinearity
	\begin{align} \label{S}
	\left\{
	\begin{array}{ccl}
	i\partial_t u -(-\Delta)^s u &=& - |x|^{-b}|u|^{p-1}u, \quad (t,x) \in \R \times \R^N,  \\
	u(0,x)  &= & u_0(x),
	\end{array}
	\right.
	\end{align}
where $N\geq 2$, $s\in(0,1)$, $p>1$ and $b>0$. The fractional Laplacian operator $(-\Delta)^s$ is defined by
$(-\Delta)^su:=\mathcal{F}^{-1}\left(|\xi|^{2s}\mathcal{F}(u)\right)$ where $\mathcal{F}$ and $\mathcal{F}^{-1}$ are the Fourier transform and  inverse Fourier transform, respectively.\\

The fractional Schr\"odinger equation arises for instance as an effective equation in the
continuum limit of discrete models with long-range interactions. In \cite{KLS13} Kirkpatrick, Lenzmann and Staffilani refer to models of mathematical biology, specifically
for the charge transport in biopolymers like the DNA. Numerous applications of fractional NLS-type equations in the physical sciences
could be mentioned, ranging from the description of Boson stars \cite{FJL07} to water
wave dynamics. The fractional Laplacian also appears as a natural operator when
considering jump processes \cite{Val09}, which makes it valuable for Lévy processes in
probability theory with applications in financial mathematics.\\

From the mathematical point of view, there is a quite large literature devoted to \eqref{S} and its variants. We emphasize that the case $s=1,\, b=0$ corresponds to the nonlinear Schr\"odinger equation ({\tt NLS}). There have been tremendous amount of researches on ({\tt NLS}), and the monographs \cite{Caz-CLNM, Fib-Book, SS99, Tao06} cover all known results in great mathematical detail.\\

The case $s=1, b\neq 0$ corresponds to the inhomegeneous nonlinear Schr\"odinger equation ({\tt INLS}). The Cauchy problem for ({\tt INLS}) with $b>0$ has received a lot of interested in the mathematical community; see, among others, \cite{GS, Guzman, Dinh, AT, KLS, Dinh-NA, Farah, BL, FG-JDE, FG-BBMS, Campos, Dinh-2D, MMZ, CFGM, DK-SIAM, Murphy} and references therein. Note that the case of spatial growing nonlinearity, that is $b<0$, has been recently investigated in \cite{DMS}. See also \cite{CG-DCDS-B, Chen, Chen-CMJ, CG-AM,ghs}. \\

We stress that the 2D energy critical counterpart has received more attention in the past decade. See, among others, \cite{CIMM, Nonlinearity, BDM-JHDE2018, BDM-CPAA2019}.\\

    Let us turn now to the fractional homogeneous case, that is, $0<s<1,\, b=0$, which we call for short ({\tt FNLS}). A first analysis of the well-posedness in Sobolev spaces has been done by Hong and Sire \cite{yhys}. It is worth noticing that the results obtained in \cite{yhys} require some extra assumptions due to the loss of $N(1-s)$ derivatives in the dispersion \cite{cox}. In the radial settings, using Strichartz estimates without loss of regularity \cite{gw1}, the local well-posedness of solutions holds in the energy space \cite{vdd}. See also \cite{gh1,gh2} for a cubic source term. The scattering of radial focusing solutions below the ground state threshold was investigated in \cite{syz}. In the energy-critical radial regime, the global well-posedness and scattering in the defocusing case, and in the focusing case with energy below the ground state, were obtained in \cite{Guo2018}. Due to the lack of a variance identity, the finite time blow-up of solutions was open for a long time. A partial result was given in \cite{bhl} by use of a localized variance identity. \\

To the best authors knowledge, there exist few works dealing with the inhomogeneous fractional Schr\"odinger equation \eqref{S}, with $b\neq0$ and $0<s<1$. Sufficient conditions for global existence in $H^s$ were derived in \cite{pz}. In addition, a blow-up criterion of radial solutions for the inter-critical regime was give in \cite{pz}. By using a sharp Gagliardo-Nirenberg inequality and potential well method, the second author obtained the well-posedness in the case of spatial growing nonlinearity, that is $b<0$; see \cite{ts}. See also \cite{st4} for the bi-harmonic Schr\"odinger equation.\\

The main goal of this paper is to fill in a gap in the literature. Indeed, the ground state threshold of global existence and scattering versus finite time blow-up of focusing radial solutions seems not to be treated before. One needs to deal with two combined difficulties. The first one is the presence of the non-local fractional Laplacian operator, which gives a loss of derivatives in the dispersive estimate \eqref{free}. The second one is the singular weight $|x|^{-b}$ in the non-linearity. The main part of this work is devoted to proving the scattering. We avoid here the Kenig-Merle road-map, based on the concentration-compactness method \cite{Merle}. Instead, we use the Dodson-Murphy method \cite{DM2017} based on Tao's scattering criteria \cite{tao1} and Morawetz estimates. Moreover, we expresse the threshold using some non-conserved quantities in the spirit of the recent paper by Dinh \cite{dd}. This gives as a consequence the classical mass-energy threshold obtained first by Roudenko et al. \cite{dhr}.\\

The non-linear Schr\"odinger equation \eqref{S} satisfies the scaling invariance
$$0< \lambda \longmapsto u_{\lambda}(t,x):=
\lambda^{\frac{2s-b}{p-1}}u(\lambda^{2s} t,\lambda x).$$
The only one homogeneous Sobolev norm invariant under the above scaling is relevant in this study. Indeed, the next identity
$$\|u_\lambda(t)\|_{\dot H^\mu}=\lambda^{\mu-(\frac N2-\frac{2s-b}{p-1})}\|u(\lambda^{2s}t)\|_{\dot H^\mu},$$
 gives rise to the critical Sobolev index $s_c:=\frac N2-\frac{2s-b}{p-1}$. The energy-critical regime corresponds to $s_c=s$ or $p=p^*:=1+\frac{2(2s-b)}{N-2s}$ and is related to the energy conservation law
\begin{equation}
    \label{Energy}
E[u(t))]:=\int_{\R^N}|\bD^s u(t)|^2\,dx - \frac2{1+p}\int_{\R^N}|x|^{-b}|u(t)|^{1+p}\,dx = E[u_0],
\end{equation}
where $\bD^{\sigma}$ stands for the operator with Fourier multiplier
\begin{equation}
\label{bD}
\mathcal{F}\left(\bD^{\sigma}\phi\right)(\xi)=|\xi|^{\sigma}\mathcal{F}({\phi})(\xi).
\end{equation}
In particular, $(-\Delta)^s=\bD^{2s}.$

The mass-critical regime corresponds to $s_c=0$, or $p=p_*:=1+\frac{2(2s-b)}{N}$ and is related to the mass conservation law
\begin{equation}
  \label{Mass}
M[u(t)]:=\int_{\R^N}|u(t,x)|^2\,dx = M[u_0].
\end{equation}
Unless otherwise specified, we restrict ourselves to the inter-critical regime $0< s_c<s$. The last condition can be written in terms of $p$ as $p_*<p<p^*$.


Define the positive real number $\gamma_c:=\frac{s-s_c}{s_c}$ and the quantities
\begin{gather}
\label{P-u}
P[\phi]:=\int_{\R^N}|x|^{-b}|\phi(x)|^{p+1}\,dx,\\
\label{I-u}
I[\phi]:=\|\bD^s\, \phi\|^2-\frac B{1+p}\int_{\R^N}|x|^{-b}|\phi(x)|^{1+p}\,dx,
\end{gather}
where $B$ is given by \eqref{AB} below and $\|\cdot\|$ denotes the $L^2(\R^N)$ norm. 
Define also the scale invariant quantities
\begin{align}
\label{M-E}
\mathcal{ME}[\phi]&:=\Big(\frac{M[\phi]}{M[Q]}\Big)^{\gamma_c}\Big(\frac{E[\phi]}{E[Q]}\Big),\\
\label{M-G}
\mathcal{MG}[\phi]&:=\Big(\frac{\|\phi\|}{\|Q\|}\Big)^{\gamma_c}\Big(\frac{\|\nabla \phi\|}{\|\nabla\,Q\|}\Big),
\end{align}
where $Q\in H^s$ is the unique non-negative radially symmetric decreasing solution of \eqref{E} given by Lemma \ref{GNI} below.

From now on, we hide the variable $t$ for simplicity, spreading it out only when necessary. Let $H_{rad}^{s}$ denote the space of functions in $H^{s}$ which are radially symmetric.
Our main contribution  reads as follows.
\begin{thm}\label{t1}
    Assume that $ N\geq 2$, $s\in(\frac N{2N-1},1)$, $0<b<2s$ and $p_*<p<p^*$. Let $u_0\in H_{rad}^{s}$ and  ${u}\in C([0,T^*); H_{rad}^{s})$  be the corresponding maximal solution of \eqref{S} given by Proposition \ref{prop0}. 
\begin{enumerate}
\item[(i)]
Suppose that 
\begin{equation}
\sup_{t\in[0,T^\ast)}P[u(t)][M[u(t)]^{\gamma_c}<P[Q][M[Q]]^{\gamma_c}.\label{ss1}
\end{equation}
Then, $u$ is global. Moreover, if $N\geq3$, $s>\frac N{1+N}$, $p> 2(1-\frac bN)$ {\rm(  $p<\frac{N-2b}{N-2s}$ if $N=3$)}, then $u$ scatters in $H^s$. 
\item[(ii)]
Suppose that
\begin{equation}
\label{bl}
\sup_{0\leq t <T^*}\,I[u(t)]<0.
\end{equation}
Then, $u$ blows-up in finite or infinite time in the sense that $T^*<\infty$ or $T^*=\infty$ and there is $t_n\to\infty$ such that $\|\bD^s\,u(t_n)\|\to\infty$ as $n$ goes to infinity.
\end{enumerate}
\end{thm}
In view of the  results stated in the above theorem, some comments are in order.
~{\rm \begin{itemize}
\item{
The scattering in $H^s$ means that there exists $\varphi_\pm\in H^s$ such that\newline  $\displaystyle\lim_{t\to\pm\infty}\|u(t)-e^{-it\bD^{2s}}\varphi_\pm\|_{H^s}=0$, where $e^{-it\bD^{2s}}$ denotes the free fractional Schr\"odinger operator given by \eqref{FO}.}
\item
In Corollary \ref{t2}, we derive the scattering under the ground state threshold as a consequence of Theorem \ref{t1}.
\item As one can see later in the proof of the above theorem, the effect  of the fractional operator appears in the scattering criteria in Section \ref{S5}. Precisely, when using the dispersive estimate \eqref{free} to handle the term $F_1$ appearing in \eqref{F1-2}. 
The natural assumption $s\geq N(1-s)$ here gives the restriction $s\geq\frac{N}{1+N}$. This is stronger than the Strichartz estimates requirement $s>\frac N{2N-1}$.
\item
It seems that with the method used here one has the scattering for some $p_*<\tilde p<p<p^*$. The same restriction was observed in \cite{syz}.
\item
The radial assumption is necessary to avoid a loss of regularity in Strichartz estimates \cite{gw1}.
\item
In the homogeneous case $b=0$, there is a scattering result under the ground state threshold of the ({\tt FNLS}), see \cite{syz}.
\item
In the non-radial case, based on the local well-posed result \cite{yhys}, the authors will treat the asymptotics of energy solutions in a forthcoming work.
\item
In the space dimension $N=2$, the inequality \eqref{false} below is false. Arguing with a different way, one obtains the scattering for $\max\{1+\frac bs,2\}\leq p\leq 1+\frac bs+(1-\frac bs)\frac N{N-2s}$; see appendix \ref{appendix1} for the details.
\item
{If one adds the supplementary assumption $p<1+4s$, then arguing as in \cite{bhl}, one can prove the finite time blow-up provided that \eqref{bl} is satisfied.
\item
In the classical case $s=1$, we can remove the radial assumption by the use of the decay of the inhomogeneous term \cite{cc}. But here, the spherically symmetric condition is also needed in Strichartz estimates.}
\item
In the attractive regime, namely the equation \eqref{S'}, the scattering of global solutions with spherically symmetric data is proved in the Appendix \ref{appendix2}.
\item
In a paper in progress, the authors investigate the scattering of \eqref{S} with $b<0$.
\end{itemize}}

As a consequence of the above result, one has the next dichotomy of global/non global existence of energy solutions under the ground state threshold.
\begin{cor}\label{t2}
Assume that $ N\geq 2$, $s\in(\frac N{2N-1},1)$, $0<b<2s$ and $p_*<p<p^*$. Let $u_0\in H_{rad}^{s}$ and  ${u}\in C([0,T^*); H_{rad}^{s})$  be the corresponding maximal solution of \eqref{S} given by Proposition \ref{prop0}. Suppose further that
\begin{equation} \label{t11}
\mathcal{ME}[u_0]<1.
\end{equation}
Then,
\begin{enumerate}
\item[(i)]
the solution $u$ is global provided that
\begin{equation}\label{t12}
\mathcal{MG}[u_0]<1.\end{equation}
Moreover, it scatters in $H^s$ if $N\geq3$, $s>\frac N{1+N}$, $2(1-\frac bN)<p$ and $p<\frac{N-2b}{N-2s}$ if $N=3$;
\item[(ii)]
the solution $u$  blows-up in finite or infinite time provided that
\begin{equation}\label{t13}
\mathcal{MG}[u_0]>1.\end{equation}
\end{enumerate}
\end{cor}
\begin{rem}
{\rm The finite time blow-up was proved in \cite{pz} under the supplementary assumption $p<1+4s$.}
\end{rem}
We conclude the introduction with an outline of the paper. In Section \ref{S2}, we recall some useful tools needed in our proofs. In addition, we give some auxiliary results in order to facilitate the reading of the rest of the paper. Section \ref{S3} is devoted to some variational analysis. We derive Morawetz-type
inequalities in Section \ref{S4}. The scattering criterion is stated and proved in Section \ref{S5}. The proof of Theorem \ref{t1} is given in Section \ref{S6}. Section \ref{S7} contains the proof of the scattering versus blow-up under the ground state threshold, that is, Corollary \ref{t2}. The two-dimential case and the defocusing regime are treated in Appendix \ref{appendix1} and Appendix \ref{appendix2}, respectively. 

Finally, the notation $A\lesssim B$ (resp. $A\gtrsim B$) for positive numbers $A$ and $B$, means that there exists a positive constant $C$ such that $A\leq C B$ (resp. $A\geq C B$).
\section{Useful tools \& Auxiliary results}
\label{S2}
\subsection{Useful tools}

For future convenience,  we recall some known and useful tools which will play an important role in the proof of our main results.

The fractional radial Sobolev inequality reads \cite[Proposition 1]{co}
\begin{equation}
\label{F-Sob-I}
\sup_{x\neq0}|x|^{\frac{N}{2}-\alpha}|u(x)|\leq C(N,\alpha)\|\bD^{\alpha} u\|,\quad\forall u\in \dot{H}^\alpha_{rad},    
\end{equation}
provided that $N\geq 2$ and $1/2<\alpha<N/2$.

We also recall the homogeneous Sobolev embedding \cite[Theorem 1.38]{BCD}
\begin{equation}
\label{Sob-s}
\|u\|_{L^{\frac{2N}{N-2\alpha}}(\R^N)}\leq C(N,\alpha) \|\bD^{\alpha}u\|,\quad\forall\, u\in \dot{H}^\alpha,
\end{equation}
provided that $0\leq \alpha<N/2$.

The classical Sobolev embedding \cite{Adams} states that 
\begin{equation}
\label{Sob-emb}
H^{\alpha}(\R^N)\hookrightarrow L^q(\R^N),\quad  2\leq \,q\leq \frac{2N}{N-2\alpha},
\end{equation}
provided that $0\leq \alpha<N/2$. 

We know from \cite[Theorem 1.62, p. 41]{BCD} that multiplication by a function of Schwartz space $\mathcal{S}(\R^N)$ is a continuous map
from $H^s(\R^N)$ into itself for any $s\in\R$. Precisely we have
\begin{equation*}
\label{SHs}
\|\varphi\,u\|_{H^s}\leq 2^{\frac{|s|}{2}}\left\|(1+|\cdot|^2)^{\frac{|s|}{2}}\widehat{\varphi}\right\|_{L^1}\,\|u\|_{H^s},
\end{equation*}
where $s\in\R$, $\varphi\in\mathcal{S}(\R^N)$ and $u\in H^s(\R^N)$.

The next fractional chain rule \cite{cw} will be useful.
\begin{lem}\label{chain}
Let $N\geq1$, $0<\alpha \leq1$,  $\frac1p=\frac1{p_i}+\frac1{q_i}$, $i=1,2$ and $F\in C^1(\C)$ . Then, 
\begin{equation}
\label{chain1}
\|\bD^{\alpha}F(u)\|_{L^p}\lesssim \|\bD^{\alpha}u\|_{L^{q_1}}\|F'(u)\|_{L^{p_1}},
\end{equation}
and 
 \begin{equation}
\label{chain2}
\|\bD^{\alpha}(uv)\|_{L^{p}}\lesssim \|\bD^{\alpha} u\|_{L^{p_1}}\|v\|_{L^{q_1}}+\|\bD^{\alpha} v\|_{L^{p_2}}\|u\|_{L^{q_2}}.
\end{equation}
\end{lem}
The next result gives a Leibniz rule for fractional derivatives \cite{kpv}.
\begin{lem}\label{lbnz}
Let $N\geq1$, $s_1+s_2:=s\in (0,1)$, such that $ s_1,s_2\in (0,1)$ and $1<p,p_1,p_2<\infty$ such that $\frac{1}{p}=\frac{1}{p_1}+\frac{1}{p_2}.$ 
Then,
\begin{equation}
\label{L-rule}
\|\bD^{s}(uv)-u\bD^{s_1}v-v\bD^{s_2}u\|_{L^{p}}\lesssim\|\bD^{s_1}v\|_{L^{p_1}}\|\bD^{s_2}u\|_{L^{p_2}}.
\end{equation}
Moreover, for $s_1=0$, the value $p_1=\infty$ is allowed.
\end{lem}
A useful Sobolev's embedding reads as follows (see \cite[Theorem 6.1.6]{BJ76}).
\begin{lem}
\label{S-emb}
Let $s_1, s_2\in\R$ and $1<p_1\leq p_2<\infty$ such that
$$
s_1-\frac{N}{p_1}=s_2-\frac{N}{p_2}.
$$
Then 
$$
W^{s_1,p_1}(\R^N)\hookrightarrow W^{s_2,p_2}(\R^N).
$$
\end{lem}
{The following continuity argument (or bootstrap argument) will also be useful for our purpose. See \cite[Lemma 2.11]{DKM} for a similar statement.
\begin{lem}\cite[Lemma 3.7, p. 437]{Strauss}\\
\label{boots}
Let $\mathbf{I}\subset\R$ be a time interval, and $\mathbf{X} : \mathbf{I}\to [0,\infty)$ be a continuous function satisfying, for every $t\in \mathbf{I}$,
\begin{equation}
		\label{boots1}
		    	\mathbf{X}(t) \leq a + b [\mathbf{X}(t)]^\theta,
		\end{equation}
	where $a,b>0$ and $\theta>1$ are constants. Assume that, for some $t_0\in \mathbf{I}$,
		\begin{equation}
		\label{boots2}
	\mathbf{X}(t_0)\leq a, \quad a\,b^{\frac{1}{\theta-1}} <(\theta-1)\,\theta^{\frac{\theta}{1-\theta}}.
				\end{equation}
		Then, for every $ t\in \mathbf{I}$, we have
		\begin{equation}
		\label{boots3}
		    	\mathbf{X}(t)< \frac{\theta\,a}{\theta-1}.
		\end{equation}
\end{lem}
\begin{proof}
We give here a simpler proof than done in \cite{Strauss}.
We argue by contradiction. Suppose that $\mathbf{X}(t_1)\geq \frac{\theta\,a}{\theta-1}$ for some $t_1 \in \mathbf{I}$. Then $\mathbf{X}(t_0)\leq a<\frac{\theta\,a}{\theta-1}\leq \mathbf{X}(t_1).$ By continuity, there exists $t_2 \in \mathbf{I}$ such that $\mathbf{X}(t_2)=\frac{\theta\,a}{\theta-1}$. This contradicts the second assumption in \eqref{boots2} since $\frac{\theta\,a}{\theta-1}\leq a+ b\bigg(\frac{\theta\,a}{\theta-1}\bigg)^\theta$ imply that  $a\,b^{\frac{1}{\theta-1}} \geq(\theta-1)\,\theta^{\frac{\theta}{1-\theta}}.$ This leads to \eqref{boots3} as desired.
\end{proof}}
The next Gagliardo-Nirenberg type inequality \cite[Theorem 2.2]{pz} is crucial in our proofs.
\begin{lem}\label{GNI}
Let $N\geq 2$, $s\in(0,1)$, $0<b<2s$ and $1<p<1+\frac{4s-2b}{N-2s}$. Then the following sharp Gagliardo-Nirenberg inequality holds
\begin{equation}
\int_{\R^N} |x|^{-b}|u|^{p+1}\,dx
\leq\, K_{opt}\,\|u\|^A\|\bD^s\,u\|^B,\label{Nirenberg}\end{equation}
where 
\begin{equation}
\label{AB}
B:=\frac 1{2s}\bigg(N(p-1) +2b\bigg)\quad\mbox{ and }\quad A:=p+1-B.
\end{equation}
 Moreover, the sharp constant is given by 
\begin{equation}\label{K-opt}
 K_{opt}=\frac{p+1}{A}\left(\frac{A}{B}\right)^{B/2}\,\|Q\|^{1-p},
\end{equation}
where $Q\in H^s$ is the unique non-negative radially symmetric decreasing solution of
\begin{equation}\label{E}
-(-\Delta)^s Q -Q +|x|^{-b}|Q|^{p-1}Q=0.
\end{equation}
Furthermore,
\begin{equation}
\label{K-opt-1}
\|\bD^s\,Q\|=\sqrt{\frac{B}{A}}\,\|Q\| \quad\text{and}\quad
\int_{\R^N}\,|x|^{-b}\,|Q(x)|^{p+1}\,dx=\frac{p+1}{A}\,\|Q\|^2.
\end{equation}
\end{lem}

Now, let us collect some standard estimates related to the Schr\"odinger equation. The free operator associated to the fractional Schr\"odinger equation is given by
\begin{equation}
\label{FO}
e^{-it\bD^{2s}}\phi:=\mathcal F^{-1}(e^{-it|\xi|^{2s}})*\phi.
\end{equation}

It is classical that \eqref{S} has the following integral formulation
$$u(t)=e^{-it\bD^{2s}}u_0+i\int_0^te^{-i(t-\tau)\bD^{2s}}[|x|^{-b}|u(\tau)|^{p-1}u(\tau)]\,d\tau.$$
The following dispersive estimate can be found in \cite{cox} for instance. 
\begin{equation}\label{free}
\|e^{-it\bD^{2s}}\,\phi\|_{L^r(\R^N)}\leq\frac C{|t|^{N(\frac12-\frac1r)}}\|\bD^{2N(1-s)(\frac12-\frac1r)}\,\phi\|_{L^{r'}(\R^N)}, \quad \forall\; r\geq2,\;\; \forall\; t\neq 0.
\end{equation}
 \begin{defi}
~\begin{enumerate}
\item[1)]
A pair $(q,r)$ is said admissible if $q,r\geq2$ and
$$\frac{4N+2}{2N-1}\leq q\leq\infty,\quad \frac2q+\frac{2N-1}r\leq N-\frac12,$$
or 
$$2\leq q\leq\frac{4N+2}{2N-1},\quad \frac2q+\frac{2N-1}r< N-\frac12.$$
\item[2)]
 $(q,r)\in\Gamma_\gamma$ if it is an admissible pair such that $(N, q, r)\neq(2,2,\infty)$ and 
$$N(\frac12-\frac1r)=\frac{2s}q+\gamma.$$
Moreover, $\Gamma:=\Gamma_0$.
\item[3)]
let $I\subset\R$ be an interval, one denotes the Strichartz spaces
$$S^\gamma(I):=\displaystyle\bigcap_{(q,r)\in\Gamma_\gamma}L^q(I,L^r),\quad S(I):=S^0(I).$$
\end{enumerate}
\end{defi}
Let us now state some Strichartz estimates \cite{gw1,cl}.
\begin{prop}\label{prop2}
Let $N \geq 2$, $\gamma\in\R$ and $u_0\in L^2_{rad}$. Then, 
\begin{enumerate}
\item[1)]
$\|e^{-it\bD^{2s}}u_0\|_{S^\gamma(I)}\lesssim\|\bD^{\gamma}u_0\|;$
\item[2)]
$\|u-e^{-it\bD^{2s}}u_0\|_{S^\gamma(I)}\lesssim\displaystyle\inf_{(\tilde q,\tilde r)\in\Gamma_{-\gamma}(I)}\|i\partial_t u-\bD^{2s}u\|_{L^{\tilde q'}(I,L^{\tilde r'})};$
\item[3)]$\|u\|_{S(I)}\lesssim \|u_0\|+\displaystyle\inf_{(\tilde q,\tilde r)\in\Gamma(I)}\|i\partial_t u-\bD^{2s}u\|_{L^{\tilde q'}(I,L^{\tilde r'})},$ provided that $\frac {N}{2N-1}<s\leq 1$.
\end{enumerate}
\end{prop}
\subsection{Auxiliary results}
\label{S2-2}
Using a contraction mapping technique based on Strichartz estimates, the Cauchy problem \eqref{S} is locally well-posed in the energy space \cite[Proposition 3]{pz}.
\begin{prop}\label{prop0}
Let $N\geq2$, $s\in(\frac{N}{2N-1},1)$, $0<b<2s$, $1< p< p^*$ and $ u_0\in  H_{rad}^s$. Then, there exist $0<T^*\leq\infty$ and a unique maximal solution ${u} \in C ([0, T^*), H_{rad}^s)$ to \eqref{S}. 
Moreover, the solution $u$ satisfies the conservation of mass and energy \eqref{Mass} and \eqref{Energy}.
\end{prop}

In trying to apply the arguments in \cite{bhl} and obtain a localized Morawetz estimate, one encounters serious difficulties due the nonlocal operator $(-\Delta)^s=\bD^{2s}$. To handle this difficulty, we use the following representation known as Balakrishnan's formula
\begin{equation}
    \label{Bala}
 \bD^{2s}=(-\Delta)^s=\frac{\sin\pi\,s}{s}\,\int_0^\infty\,m^{s-1}\,\bigg(\frac{-\Delta}{-\Delta+m}\bigg)\,dm.   
\end{equation}
We define the auxiliary function $u_m$, for $m>0$, by
\begin{equation}
\label{u-m}
u_m:=c_s\,R_m\,u
:=c_s(m-\Delta)^{-1}\,u:=c_s\,{\mathcal F}^{-1}\bigg(\frac{\what{u}(\xi)}{m+|\xi|^2}\bigg),
\end{equation}
where
\begin{equation*}
\label{cs}
c_s=\sqrt{\frac{\sin\pi\,s}{\pi}}.
\end{equation*}
The following formula can be easily derived by using Plancherel's and Fubini's theorem as in \cite[(2.12)]{bhl}
\begin{equation}
\label{2-12}
    s\|\bD^{s}u\|^2=\int_0^\infty m^s\int_{\R^N}|\nabla u_m|^2\,dx\,dm.
\end{equation}
Here and hereafter, we denote by $B(R):=\{x\in\R^N;\;|x|\leq R\}$ the ball of $\R^N$ centered at the origin and with radius $R>0$. Let $\psi\in C_0^\infty(\R^N)$ be a radial bump function such that
\begin{equation}
 \label{bump-psi} 
 \psi=1\quad\mbox{on}\quad B\left(\frac12\right),\quad\psi=0 \quad \mbox{for}\quad|x|\geq 1\quad\mbox{and}\quad 0\leq\psi\leq 1.
\end{equation}
For $R>0$, define
\begin{equation}
\label{psi-R}
\psi_{R}(x)=\psi\left(\frac{|x|}{R}\right).
\end{equation}
\begin{lem}
\label{IBP}
Let $\varphi\in C^\infty_0(\R^N)$ be a real-valued function and $v\in H^1(\R^N)$. Then
\begin{equation}
    \label{IBP1}
    \int_{\R^N}\,|\nabla(\varphi\,v)|^2\,dx=\int_{\R^N}\,\varphi^2|\nabla\,v|^2\,dx-\int_{\R^N}\,\varphi\,\Delta\varphi\,|v|^2\,dx,
\end{equation}
\begin{eqnarray}
\nonumber
|\nabla(\varphi\,v)_m|^2&=&|\nabla(\varphi\,v_m)|^2-
\nabla(\varphi\,v_m)\cdot\nabla(R_m[-\Delta,\varphi]\bar {v}_m)\\
\label{phivm}&-&\nabla(R_m[-\Delta,\varphi]v_m)\cdot\nabla(\varphi\bar{v})_m
\end{eqnarray}
and
\begin{eqnarray}
   \nonumber
   s\|\bD^{s}(\varphi\,v)\|^2
   &=&{ \int_0^\infty\int_{\R^N} m^s|\nabla(\varphi\,v)_m|^2\,dx\,dm-\int_0^\infty\int_{\R^N} m^s|\nabla(\varphi\,v_m)|^2\,dx\,dm}\\
    \label{loc-frac-1}
   &+& \int_0^\infty\,\int_{\R^N}\,m^s\,\varphi^2|\nabla\,v_m|^2\,dx\,dm-\int_0^\infty\,\int_{\R^N}\,m^s\,\varphi\,\Delta\varphi\,|v_m|^2\,dx\,dm.
\end{eqnarray}
\end{lem}
\begin{proof}
 The proof of \eqref{IBP1} is straightforward and uses integration by parts. In particular, we obtain the useful identity
 \begin{eqnarray}
 \nonumber
 \int_0^\infty\int_{\R^N} m^s|\nabla(\varphi\,v_m)|^2\,dx\,dm&=&\int_0^\infty\int_{\R^N} m^s\varphi^2|\nabla v_m|^2\,dx\,dm\\
 \label{phivm1}
 &-&\int_0^\infty\int_{\R^N}m^s\varphi\Delta\varphi|v_m|^2\,dx\,dm.
 \end{eqnarray}
 To prove \eqref{phivm}, let us write using computation done in the proof of \cite[Lemma 4.8]{syz},
\begin{eqnarray*}
\nabla(\varphi\,v)_m
&=&\nabla(\varphi\,v_m)-\nabla[\varphi,R_m]v\\
&=&\nabla(\varphi\,v_m)-\nabla(R_m[-\Delta+m,\varphi]R_m\,v)\\
&=&\nabla(\varphi\,v_m)-\nabla(R_m[-\Delta,\varphi]v_m).
\end{eqnarray*}
Hence
\begin{eqnarray}
\label{nablaphivm}
|\nabla(\varphi\,v)_m|^2
&=&\nabla(\varphi\bar {v})_m\Big(\nabla(\varphi\,v_m)-\nabla(R_m[-\Delta,\varphi]v_m)\Big)\\ \nonumber
&=&\nabla(\varphi\,v_m)\Big(\nabla(\varphi\bar{v}_m)-\nabla(R_m[-\Delta,\varphi]\bar{v}_m)\Big)-\nabla(R_m[-\Delta,\varphi]v_m)\nabla(\varphi\bar{v})_m\\
\nonumber
&=&|\nabla(\varphi\,v_m)|^2-\nabla(\varphi\,v_m)\nabla(R_m[-\Delta,\varphi]\bar{v}_m)-\nabla(R_m[-\Delta,\varphi]v_m)\nabla(\varphi\bar{v})_m.
\end{eqnarray}
Finally, the proof of \eqref{loc-frac-1} follows easily from \eqref{phivm} and \eqref{phivm1}.
\end{proof}

Now, let us give an improvement of \cite[Lemma 4.2]{syz}.
\begin{lem}\label{improv4.2}
Let $N\geq 2$, $\theta\in (0,1)$ and $u\in L^2(\R^N)$. Then 
\begin{equation}
\label{R-theta}
\bigg|\int_0^\infty\int_{\R^N}\,m^{\theta}\,\psi_R\Delta\psi_R|u_m|^2\,dx\,dm\bigg|\leq C\,R^{-2\theta},
\end{equation}
where $u_m$ given by \eqref{u-m} and $C=C({N},s,\theta, \psi, \|u\|)$ is a positive constant depending only on $N, s, \theta, \psi$ and $\|u\|$.
\end{lem}
\begin{proof}
Note that, for any $0\leq \alpha<2$, we have
\begin{equation}
\label{um-alpha}
\|\bD^{\alpha}u_m\|\lesssim\, m^{\frac{\alpha}{2}-1}\,\|u\|.
\end{equation}
Indeed, by \eqref{u-m}, we get
\begin{eqnarray*}
\|\bD^{\alpha}u_m\|&=&c_s\bigg\|\frac{|\xi|^\alpha}{m+|\xi|^2}\,\what{u}(\xi)\bigg\|\\
&\leq&{C_{N,s}}\bigg\|\frac{t^\alpha}{m+t^2}\bigg\|
_{L^\infty(\R_+)}\,\|u\|\\
&\leq&{C_{N,s}} m^{\frac{\alpha}{2}-1}\,\|u\|.
\end{eqnarray*}
Next, for $M>0$ we write 
$$
\bigg|\int_0^\infty\int_{\R^N}\,m^{\theta}\,\psi_R\Delta\psi_R|u_m|^2\,dx\,dm\bigg|{\leq}(I)+(II),
$$
where
\begin{eqnarray*}
(I)&:=&\bigg|\int_0^M\,\int_{\R^N}\,m^{\theta}\,\psi_R\Delta\psi_R|u_m|^2\,dx\,dm\bigg|,\\
(II)&:=&\bigg|\int_M^\infty\,\int_{\R^N}\,m^{\theta}\,\psi_R\Delta\psi_R|u_m|^2\,dx\,dm\bigg|.
\end{eqnarray*}
As in \cite[Lemma 4.2]{syz}, we have 
$$
(II)\lesssim R^{-2}\, M^{\theta-1}.
$$
To handle the first term $(I)$ let us choose {$\alpha=1-\frac{\theta}{2}$} and write by H\"older's inequality, \eqref{Sob-s} and \eqref{um-alpha},
\begin{eqnarray*}
(I)&\leq&\int_0^M\,m^{\theta}\|\psi_R\|_{\frac {N}{2\alpha}}\|\Delta\psi_R\|_\infty\|u_m\|^2_{\frac{2N}{N-2\alpha}}\,dm\\
&\lesssim& R^{2\alpha-2}\int_0^M\,m^{\theta}\,m^{\alpha-2}\,dm\\
&\lesssim&R^{-\theta}\int_0^M\,m^{\frac{\theta}{2}-1}\,dm\\
&\lesssim& R^{-\theta}\, M^{\frac{\theta}{2}}.
\end{eqnarray*}
We conclude the proof by choosing $M=R^{-2}$.
\end{proof}
We also need the following estimates.
\begin{lem}
\label{1-R}
Let $u\in H^s(\R^N)$ with $s\in (\frac{1}{2},1)$ and $\theta\in (0,s)$. Then,
\begin{equation}
\label{e1}
\int_0^\infty\int_{\R^N}\,m^{\theta}\,|\Delta\psi_R|^2|u_m|^2\,dx\,dm\lesssim\frac{1}{R^2},
\end{equation}
\begin{equation}
\label{e2}
\|\psi_R\,u\|_{\dot{H}^{\theta}}\leq\, C(s,\psi)\,\|u\|_{H^s},\quad R\geq 1,
\end{equation}
and
\begin{equation}
\label{e3}
\int_0^\infty\int_{\R^N}\,m^{\theta}\,|\nabla\psi_R|^2|\nabla\,u_m|^2\,dx\,dm\lesssim\frac{1}{R^2}.
\end{equation}
\end{lem}
\begin{proof}
 First, let us prove \eqref{e1}. For $M>0$ we write
 \begin{eqnarray*}
 \int_0^\infty\int_{\R^N}\,m^{\theta}\,|\Delta\psi_R|^2|u_m|^2\,dx\,dm&=&\int_0^M\int_{\R^N}\,m^{\theta}\,|\Delta\psi_R|^2|u_m|^2\,dx\,dm\\&+&\int_M^\infty\int_{\R^N}\,m^{\theta}\,|\Delta\psi_R|^2|u_m|^2\,dx\,dm.
 \end{eqnarray*}
 The second term in the RHS can be estimated as follows
 \begin{eqnarray*}
 \int_M^\infty\int_{\R^N}\,m^{\theta}\,|\Delta\psi_R|^2|u_m|^2\,dx\,dm&\leq&\int_M^\infty\,m^{\theta}\,\|\Delta\psi_R\|_\infty^2\|u_m\|^2\,dm\\
 &\lesssim&R^{-4}\,\int_M^\infty\,m^{\theta-2}\,dm\\
 &\lesssim& R^{-4}\,M^{\theta-1}.
 \end{eqnarray*}
 To estimate the first term in the RHS, we use H\"older's inequality together with \eqref{Sob-s} and \eqref{um-alpha} to obtain
 \begin{eqnarray*}
 \int_0^M\int_{\R^N}\,m^{\theta}\,|\Delta\psi_R|^2|u_m|^2\,dx\,dm&\leq&\int_0^M\,m^{\theta}\,\|\Delta\psi_R\|_{\frac{N}{2\alpha}}\|\Delta\psi_R\|_\infty\,\|u_m\|_{\frac{2N}{N-2\alpha}}^2\,dm\\
 &\lesssim& R^{2\alpha-4}\,\int_0^M\,m^{\theta}\,m^{\alpha-2}\,dm\\
 &\lesssim& R^{-\theta-2}\,M^{\theta/2},
 \end{eqnarray*}
 where $\alpha=1-\theta/2$. Choosing $M=R^{-2}$ yields \eqref{e1}.
 
 Next, we turn to \eqref{e2}. Since $\theta \in (0,s)$ then $\|\psi_R u\|_{\dot{H}^{\theta}}\leq\|\psi_R u\|_{H^s}$. Using \eqref{SHs}, we conclude the proof of \eqref{e2}.
 
 Finally, let us prove \eqref{e3}. By \eqref{2-12} and the fact that $\theta \in (0,s)$, we have 
 $$
 \int_0^\infty\int_{\R^N}\,m^{\theta}\,|\nabla\psi_R|^2|\nabla\,u_m|^2\,dx\,dm\lesssim \frac{1}{R^2}\,\|u\|_{\dot{H}^{\theta}}^2\lesssim \frac{1}{R^2}\,\|u\|_{{H}^{s}}^2.
 $$
\end{proof}
\begin{lem}
\label{E123}
Let $u\in H^s$ with $s\in(1/2,1)$. Then
\begin{equation}
\label{Ee1}
\bigg|\int_0^\infty\int_{\R^N}\,m^{s}\,\psi_R\Delta\psi_R|u_m|^2\,dx\,dm\bigg|\lesssim \frac{1}{R},
\end{equation}
\begin{equation}
\label{Ee2}
\bigg|\int_0^\infty\int_{\R^N}\,m^{s}\, \nabla\left(R_m[-\Delta, \psi_R]u_m]\right)\cdot\nabla(\psi_R \bar{u})_m\,dx\,dm \bigg|\lesssim \frac{1}{R},
\end{equation}
and
\begin{equation}
\label{Ee3}
\bigg|\int_0^\infty\int_{\R^N}\,m^{s}\, \nabla(\psi_R\,u_m)\cdot\nabla\left(R_m[-\Delta, \psi_R]\bar{u}_m]\right)\,dx\,dm \bigg|\lesssim \frac{1}{R}.
\end{equation}
\end{lem}

\begin{proof}
 The estimate \eqref{Ee1} follows easily from \eqref{R-theta} and $2s>1$.
 
 The proof of \eqref{Ee2} uses similar arguments as in \cite[Lemma 4.3]{syz} together with \eqref{e1}-\eqref{e2}-\eqref{e3}.
 
 To prove \eqref{Ee3} let us first define
 $$
 I(R):=\bigg|\int_0^\infty\int_{\R^N}\,m^{s}\, \nabla(\psi_R\,u_m)\cdot\nabla\left(R_m[-\Delta, \psi_R]\bar{u}_m\right)\,dx\,dm \bigg|.
 $$
Note that
$$
[-\Delta, \psi_R]\bar{u}_m=-(\Delta\psi_R)\bar{u}_m-2\nabla\psi_R\cdot\nabla\bar{u}_m,
$$
and
$$
\|\nabla(R_m w)\|\lesssim\,m^{-1/2}\,\|w\|.
$$
Taking $\beta\in(\frac1{2s},1)$ and noticing that $2s(1-\beta)\in (0,1)$, $2s\beta-1\in (0,s)$, one writes thanks to Cauchy-Schwarz's inequality
\begin{eqnarray*}
I(R)&\lesssim&\|m^{s\beta-\frac12}(\Delta(\psi_R)u_m+2\nabla\psi_R\cdot\nabla u_m)\|_{L^2(\R_+,L^2(\R^N))}\,\|m^{s(1-\beta)}\nabla(\psi_Ru_m)\|_{L^2(\R_+,L^2(\R^N))}\\
&\lesssim&\frac{1}{R}\,\|m^{s(1-\beta)}\nabla(\psi_Ru_m)\|_{L^2(\R_+,L^2(\R^N))},
\end{eqnarray*}
where we have used \eqref{e1} and \eqref{e2}.
Now, applying Lemma \ref{improv4.2}, one has
\begin{eqnarray*}
\|m^{s(1-\beta)}\nabla(\psi_Ru_m)\|_{L^2(\R_+,L^2(\R^N))}^2
&=&\int_0^\infty m^{2s(1-\beta)}\int_{\R^N}|\nabla(\psi_Ru_m)|^2\,dx\,dm\\
&=&\int_0^\infty m^{2s(1-\beta)}\int_{\R^N}(\psi_R^2|\nabla u_m|^2-\psi_R\Delta\psi_R|u_m|^2)\,dx\,dm\\
&\lesssim&\|u\|^2_{H^{s(1-\beta)}}+\frac1{R^{4s(1-\beta)}}\\
&\lesssim&\|u\|^2_{H^s}+1.
\end{eqnarray*}
This gives \eqref{Ee3} as desired.
\end{proof}
As a consequence, we obtain
\begin{lem}
\label{locfrac}
Let $u\in H^s(\R^N)$. Then, as $R\to\infty$, we have
\begin{eqnarray}
\nonumber
s\|\bD^{s}(\psi_R\,u)\|^2
&\leq&\int_0^\infty\,\int_{\R^N}\,m^s\, \psi^2_R\,|\nabla\,u_m|^2\,dx\, dm+\text{O}\left(\frac1R\right)\\
\label{psiRu}
&\leq&s\|\bD^{s}u\|^2+\text{O}\left(\frac1R\right).
\end{eqnarray}
\end{lem}
\begin{proof}
Using \eqref{loc-frac-1} and \eqref{nablaphivm}, we have
\begin{eqnarray*}
s\|\bD^{s}(\psi_Ru)\|^2
&=&\int_0^\infty\int_{\R^N} m^s\,\psi_R^2\,|\nabla u_m|^2\,dx\,dm-\int_0^\infty\int_{\R^N}\,m^s\,\psi_R\Delta\psi_R|u_m|^2\,dx\,dm\\
&-&\int_0^\infty\int_{\R^N}\, m^s\,\nabla(\psi_R\,u_m)\cdot\nabla(R_m[-\Delta,\psi_R]\bar{u}_m)\,dx\,dm\\
&-&\int_0^\infty\int_{\R^N}\, m^s\,\nabla(R_m[-\Delta,\psi_R]u_m)\cdot\nabla(\psi_R\bar{u})_m\,dx\,dm.
\end{eqnarray*}
It follows from the estimates \eqref{Ee1}-\eqref{Ee2}-\eqref{Ee3}  that
\begin{eqnarray*}
s\|\bD^{s}(\psi_R\,u)\|^2
&\leq&\int_0^\infty\,\int_{\R^N}\,m^s\, \psi^2_R\,|\nabla\,u_m|^2\,dx\, dm+\text{O}\left(\frac1R\right)\\
&\leq&s\|\bD^{s}u\|^2+\text{O}\left(\frac1R\right),\;\;\;\text{as}\;\;\; R\to\infty.
\end{eqnarray*}
 This finishes the proof of Lemma \ref{locfrac}.
 \end{proof}



\section{Variational analysis}
\label{S3}

Recall that $Q$ stands for the unique nonnegative radially symmetric decreasing solution to \eqref{E}. The following inequality will be useful in obtaining a coercivity result.
\begin{lem}
\label{Coe}
Let $u\in H^s$. Then
\begin{equation}
\label{Coer0}
P[u]
\leq\frac{p+1}{B}\bigg(\frac{M[u]^{\gamma_c}P[u]}{M[Q]^{\gamma_c}P[Q]}\bigg)^{\frac{B-2}{B}}\;\|\bD^{s} u\|^2,
\end{equation}
where $B$ is given by \eqref{AB}.
\end{lem}
\begin{proof}
Thanks to Pohozaev identities \eqref{K-opt-1} (see also \cite[Theorem 2.2]{pz}), one has
$$P[Q]=\frac{p+1}{A}\,M[Q]=\frac{p+1}{B}\,\|\bD^{s}Q\|^2.$$
Using the Gagliardo-Nirenberg inequality \eqref{Nirenberg}, the expression of $K_{opt}$ given by \eqref{K-opt} and the identities $(p-1)s_c=s(B-2)$ and $\gamma_c(B-2)=A$, one writes
\begin{eqnarray*}
[P[u]]^\frac{B}2
&\leq& K_{opt}\left(\|u\|^{2\gamma_c}P[u]\right)^{\frac B2-1}\,\|\bD^{s} u\|^B\\
&\leq& \frac{p+1}{A}\left(\frac AB\right)^{\frac{B}2}\|Q\|^{-(p-1)}\left(M[u]^{\gamma_c}P[u]\right)^{\frac B2-1}\|\bD^{s} u\|^B\\
&\leq& \frac{1+p}A\left(\frac AB\right)^{\frac{B}2}M[Q]^{\frac{A-(p-1)}2}[P[Q]]^{\frac B2-1}\bigg(\frac{M[u]^{\gamma_c}P[u]}{M[Q]^{\gamma_c}P[Q]}\bigg)^{\frac B2-1}\|\bD^{s} u\|^B\\
&\leq&\bigg(\frac AB\frac{P[Q]}{M[Q]}\bigg)^{\frac{B}2}\,\bigg(\frac{M[u]^{\gamma_c}P[u]}{M[Q]^{\gamma_c}P[Q]}\bigg)^{\frac B2-1}\,\|\bD^{2s} u\|^B\\
&\leq&\bigg(\frac{M[u]^{\gamma_c}P[u]}{M[Q]^{\gamma_c}P[Q]}\bigg)^{\frac B2-1}\,\bigg(\frac{1+p}B\|\bD^{s} u\|^2\bigg)^{\frac{B}2}.
\end{eqnarray*}
This leads to \eqref{Coer0} as desired.
\end{proof}
As a consequence of the above lemma, we obtain the following coercivity result.
\begin{cor}\label{bnd}
Let $u\in H^s$ and $\varepsilon\in (0,1)$ satisfying
\begin{equation}\label{1}
P[u][M[u]]^{\gamma_c}\leq(1-\varepsilon)P[Q][M[Q]]^{\gamma_c}.
\end{equation}
Then, 
\begin{equation}
\label{Coer1}
P[u]\leq \frac{p+1}{B}\;(1-\varepsilon)^{\frac{B-2}{B}}\;\|\bD^{s}u\|^2\leq \frac{p+1}{B}\;\|\bD^{s} u\|^2,
\end{equation}
and 
\begin{equation}
\label{Coer2}
\|\bD^{s} u\|^2-\frac{B}{p+1}P[u]\geq c(\varepsilon, B)\;\|\bD^{s} u\|^2,
\end{equation}
where $c(\varepsilon ,B):= 1-(1-\varepsilon)^{\frac{B-2}{B}}>0$.
Moreover, for $\varepsilon$ small enough, we have
\begin{equation}
\label{Coer3}
E[u]\geq \frac{B-2}{B}\;\|\bD^{s} u\|^2.
\end{equation}
\end{cor}

\begin{proof}
Inequality \eqref{Coer1} follows immediately from \eqref{Coer0} and \eqref{1}.
To prove \eqref{Coer2} we use the first inequality in \eqref{Coer1} and the fact that $B>2$ and $\varepsilon\in (0,1)$.
Finally, using the first inequality in \eqref{Coer1}, we infer
\begin{eqnarray*}
E[u]
&=&\|\bD^{s} u\|^2-\frac2{1+p}P[u]\\
&\geq&\bigg(1-\frac{2}B\Big(1-\varepsilon\Big)^\frac{B-2}B\bigg)\|\bD^{s} u\|^2 .
\end{eqnarray*}
Since $1-\frac{2}B\Big(1-\varepsilon\Big)^\frac{B-2}B\to 1-\frac{2}{B}$ as $\varepsilon\to 0$ and $B>2$, we get \eqref{Coer3}.
\end{proof}
\begin{rem}
{\rm Since $$P[\psi_R\,u]\leq P[u]\quad\mbox{and}\quad M[\psi_R\,u]\leq M[u], \;\;\; \forall\;\;\;R>0,$$
inequalities \eqref{Coer1}-\eqref{Coer2} remain true for $\psi_R\,u$ instead of $u$. Namely, we have
\begin{equation}
\label{Coer1-R}
P[\psi_R\,u]\leq \frac{p+1}{B}\;\|\bD^{s}(\psi_R\,u)\|^2,
\end{equation}
and
\begin{equation}
\label{Coer2-R}
\|\bD^{s}(\psi_R\,u)\|^2-\frac{B}{p+1}P[\psi_R\,u]\geq c(\varepsilon, B)\;\|\bD^{s} (\psi_R\,u)\|^2.
\end{equation}}
\end{rem}
\begin{rem}{\rm The solution is global by \eqref{Coer3}.}
\end{rem}
\section{Morawetz estimates}
\label{S4}
In this section, we assume that $N\geq 2$, $s\in (\frac{N}{2N-1},1)$, $b\in (0, 2s)$, $p\in (p_*, p^*)$ and $u\in C([0,\infty); H^s(\R^N))$ is a global solution of \eqref{S}. \\

Consider a smooth real-valued function $f$ such that
\begin{equation}
\label{f}
0\leq f''\leq1\quad\text{and}\quad
f(r)=\left\{
\begin{array}{ll}
\frac{r^2}2,\quad\mbox{if}\quad 0\leq r\leq1;\\
1,\quad\mbox{if}\quad  r\geq2.
\end{array}
\right.
\end{equation}

As a consequence, we see that $f'(r)\leq r$.
Define, for $R>0$, the smooth radial function on $\R^N$ by
\begin{equation}
\label{f-R}
f_R(x):=R^2f\left(\frac{|x|}{R}\right).
\end{equation} 
We easily deduce from \eqref{f} and \eqref{f-R} the following properties
\begin{equation}
\label{ff-R}
0\leq f_R''\leq1,\quad f_R'(r)\leq r,\quad N-\Delta f_R\geq 0.
\end{equation}

Denote the localized virial of $u(t,x)$ by
\begin{equation}
\label{M-R}
M_{R}[u(t)]:=2\Im\int_{\R^N}\bar {u}(t)\nabla{f_R}\cdot\nabla u(t)\,dx := 2\Im\int_{\R^N}\bar u\partial_k{f_R}\partial_k u\,dx. 
\end{equation}
The next lemma gives the evolution of $M_{R}[u(t)]$.
\begin{lem}
\label{Evo-M-R}
We have
\begin{eqnarray}
\nonumber
\frac d{dt} M_{R} [u(t)]
&=&4\int_0^\infty m^s\int_{|x|<R}|\nabla u_m|^2\,dx\,dm+4\int_0^\infty m^s\int_{R<|x|<2R}{f''}\left(\frac{|x|}{R}\right)|\nabla u_m|^2\,dx\,dm\\
\label{E-M-R}
&-&\int_0^\infty m^s\int_{\R^N}\Delta^2{f_R}|u_m|^2\,dx\,dm-\frac{4sB}{1+p}\int_{|x|<R}|x|^{-b}|u|^{p+1}\,dx\\
\nonumber
&+&\frac{2(p-1)}{1+p}\int_{|x|>R}(N-\Delta{f}_R)|x|^{-b}|u|^{p+1}\,dx-\frac{4b}{1+p}\int_{|x|>R}\frac{x\cdot\nabla{f}_R}{|x|^2}|x|^{-b}|u|^{1+p}\,dx,
\end{eqnarray}
where $u_m$ is given by \eqref{u-m}.
\end{lem}
\begin{proof}
 Let us denote the differential operator $\mathbf{\Gamma}_{R}$ by
$$\mathbf{\Gamma}_{R}\,\phi  =-i\bigg(\textnormal{div}(\phi\nabla{f_R}) + \nabla{f_R}\cdot\nabla \phi\bigg),$$
which satisfies
$$<u(t),\mathbf{\Gamma}_{R}\, u(t)>=M_{R}[u(t)] .$$

Using \eqref{S}, the time derivative of $M_{R}[u(t)]$ reads 
\begin{eqnarray*}
\frac d{dt} M_{R} [u(t)] &=&<u(t),[\bD^{2s} ,i\mathbf{\Gamma}_{R} ]u(t)>+<u(t),[-|x|^{-b}|u|^{p-1},i\mathbf{\Gamma}_{R} ]u(t)>\\
&:=&\mathbf{A}(t)+\mathbf{B}(t),
\end{eqnarray*}
where the commutator of $X$ and $Y$ is $[X, Y ]:= XY -Y X$. Thanks to the computations done in \cite{bhl}, one writes
\begin{eqnarray*}
\mathbf{A}(t)&=&\int_0^\infty m^s\int_{\R^N}\Big(4\partial_k\bar u_m\partial^2_{kl}{f}_R\partial_lu_m-\Delta^2{f}_R|u_m|^2\,\Big)dx\,dm\\
&=&4\int_0^\infty m^s\int_{|x|<R}|\nabla u_m|^2\,dx\,dm+4\int_0^\infty m^s\int_{R<|x|<2R}{f''}\left(\frac{|x|}{R}\right)|\nabla u_m|^2\,dx\,dm\\
&-&\int_0^\infty m^s\int_{\R^N}\Delta^2{f_R}|u_m|^2\,dx\,dm.
\end{eqnarray*}
Let us denote the source term $\mathcal{N}:=|x|^{-b}|u|^{p-1}u$ and compute
\begin{eqnarray*}
\mathbf{B}(t)&=&-<u,\frac{\mathcal N}u\nabla{f}_R\cdot\nabla u>-<u,\frac{\mathcal N}u \textnormal{div}(u\nabla{f}_R)>\\
&+&<u,\nabla{f}_R\nabla\mathcal N>+<u,\textnormal{div}[\mathcal N\nabla{f}_R]>\\
&=&-2<u,\frac{\mathcal N}u\nabla{f}_R\nabla u>+2<u,\nabla{f}_R\nabla\mathcal N>\\
&=&2\int_{\R^N}|u|^2\nabla{f}_R\nabla[\frac{\mathcal N}u]\,dx.
\end{eqnarray*}
Integrating by parts yields
\begin{eqnarray*}
\mathbf{B}(t)
&=&2\int_{\R^N}|u|^2\nabla{f}_R\nabla[|x|^{-b}|u|^{p-1}]\,dx\\
&=&-2\int_{\R^N}[\nabla(|u|^2)\nabla{f}_R+|u|^2\Delta{f}_R]|x|^{-b}|u|^{p-1}\,dx\\
&=&-2\int_{\R^N}\nabla(|u|^2)\nabla{f}_R|x|^{-b}|u|^{p-1}\,dx-2\int_{\R^N}\Delta{f}_R|x|^{-b}|u|^{p+1}\,dx\\
&=&-2\int_{\R^N}\Delta{f}_R|x|^{-b}|u|^{p+1}\,dx-\frac4{1+p}\int_{\R^N}|x|^{-b}\nabla{f}_R\nabla(|u|^{1+p})\,dx.
\end{eqnarray*}
Hence,
\begin{eqnarray*}
\mathbf{B}(t)
&=&-\frac{2(p-1)}{1+p}\int_{\R^N}\Delta{f}_R|x|^{-b}|u|^{p+1}\,dx-\frac{4b}{1+p}\int_{\R^N}\frac{x\cdot\nabla{f}_R}{|x|^2}|x|^{-b}|u|^{1+p}\,dx\\
&=&-\frac{2N(p-1)}{1+p}\int_{\R^N}|x|^{-b}|u|^{p+1}\,dx+\frac{2(p-1)}{1+p}\int_{|x|>R}(N-\Delta{f}_R)|x|^{-b}|u|^{p+1}\,dx\\
&-&\frac{4b}{1+p}\int_{\R^N}\frac{x\cdot\nabla{f}_R}{|x|^2}|x|^{-b}|u|^{1+p}\,dx\\
&=&-\frac{4sB}{1+p}\int_{|x|<R}|x|^{-b}|u|^{p+1}\,dx+\frac{2(p-1)}{1+p}\int_{|x|>R}(N-\Delta{f}_R)|x|^{-b}|u|^{p+1}\,dx\\
&-&\frac{4b}{1+p}\int_{|x|>R}\frac{x\cdot\nabla{f}_R}{|x|^2}|x|^{-b}|u|^{1+p}\,dx.
\end{eqnarray*}
Thus, with the above calculus we obtain \eqref{E-M-R} as desired.
\end{proof}
\begin{cor}
\label{bnd1}
We have
\begin{equation}
\label{Bound1}
    \int_0^T\int_{|x|<R}|x|^{-b}|u(t,x)|^{1+p}\,dx\,dt\lesssim\,R+TR^{-b},\quad\mbox{for any}\quad T,R>0.
\end{equation}
\end{cor}
\begin{proof}
By \cite[Lemma A.2]{bhl}, one has
$$\bigg|\int_0^\infty m^s\int_{\R^N}\Delta^2{f_R}|u_m|^2\,dx\,dm\bigg|\lesssim R^{-2s}.$$
Thus, using the properties of $f_R$ and the fact that $b<2s$, one writes
\begin{eqnarray*}
\frac d{dt} M_{R} [u(t)]
&\geq&4\int_0^\infty m^s\int_{|x|<R}|\nabla u_m|^2\,dx\,dm-\frac{4sB}{1+p}\int_{|x|<R}|x|^{-b}|u|^{1+p}\,dx+O\Big(R^{-b}\Big).
\end{eqnarray*}
Taking into account Corollary \ref{bnd} and Lemma \ref{locfrac}, with the fact that $0\leq\psi\leq1$ and $\psi=0$ on $B(\frac12)$, one gets
\begin{eqnarray*}
\frac d{dt} M_{R} [u(t)]
&\geq&4\int_0^\infty m^s\int_{|x|<R}|\nabla u_m|^2\,dx\,dm-\frac{4sB}{1+p}\int_{|x|<R}|x|^{-b}|u|^{1+p}\,dx+O\Big(R^{-b}\Big)\\
&\geq&4\int_0^\infty m^s\int_{\R^N}\psi_R^2|\nabla u_m|^2\,dx\,dm+4\int_0^\infty m^s\int_{R/2<|x|<R}(1-\psi_R^2)|\nabla u_m|^2\,dx\,dm\\
&&-\frac{4sB}{1+p}P[\psi_R u]-\frac{4sB}{1+p}\int_{R/2<|x|<R}|x|^{-b}(1-\psi_R^{p+1})|u|^{1+p}\,dx+O\Big(R^{-b}\Big)\\
&\geq&4s\|\bD^{s}(\psi_R u)\|^2-\frac{4sB}{1+p}P[\psi_R u]+O\Big(R^{-b}\Big)\\
&\geq&cP[\psi_Ru]+O\Big(R^{-b}\Big).
\end{eqnarray*}
Hence,
\begin{eqnarray*}
\sup_{t\in[0,T]} |M_{R} [u(t)]|
&\gtrsim&\int_0^TP[\psi_Ru]\,ds+O\Big(R^{-b}\Big)T\\
&\gtrsim&\int_0^T\int_{|x|<\frac R2}|x|^{-b}|u|^{1+p}\,dx\,ds+O\Big(R^{-b}\Big)T.
\end{eqnarray*}
Thus, with previous computation via \cite[Lemma A.1]{bhl} and the assumption $s>\frac12$, one gets
\begin{eqnarray*}
\int_0^T\int_{|x|<\frac R2}|x|^{-b}|u|^{1+p}\,dx\,dt
&\leq& C\Big(\sup_{[0,T]}|M_R[u(t)]|+TR^{-b}\Big)\\
&\leq& C\Big(R+TR^{-b}\Big).
\end{eqnarray*}

\end{proof}
\begin{cor}
\label{bnd2}
For any sequence $R_n\to\infty$, there exists a sequence $t_n\to\infty$ such that 
\begin{equation}
\label{Bound2}
\displaystyle\lim_{n\to\infty}\,\int_{|x|<R_n}|x|^{-b}|u(t_n,x)|^{1+p}\,dx=0.
\end{equation}
\end{cor}
\begin{proof}
Let $(R_n)$ be a sequence of positive numbers tending to infinity. By taking $T_n=R_n^{1+b}$, it follows from \eqref{Bound1}
that
$$\frac{2}{T_n}\int_{T_n/2}^{T_n}\int_{|x|<{R_n}}\,|x|^{-b}|u(t,x)|^{1+p}\,dx\,dt\lesssim R_n^{-b} \to0\;\;\text{as}\;\; n\to\infty.$$
The proof follows using the integral mean value theorem.
\end{proof}
\section{Scattering criterion}
\label{S5}
Here and hereafter, one denotes the real numbers
\begin{gather*}
a:=\frac{s(1+p-\theta)}{s-s_c},\quad d:=\frac{s(1+p-\theta)}{s+(p-\theta)s_c};\\
{q:=\frac{2s(1+p-\theta)}{2(s-s_c)+s_c(1+p-\theta)}};\\
r:=\frac{2N(1+p-\theta)}{(N-2s_c)(1+p-\theta)-4(s-s_c)}.
\end{gather*}
{Clearly, one can choose $\theta>0$  small enough so that}
$$(q,r)\in\Gamma,\quad  (a,r)\in \Gamma_{s_c},\quad (d,r)\in \Gamma_{-s_c}\quad\mbox{and}\quad (p-\theta)d'=a.$$

This section is devoted to the proof of the following scattering criterion. 
\begin{prop}\label{crt}
Take the assumptions of Theorem \ref{t1}. Let $u\in C(\R,H^s_{rad})$ be a global solution to \eqref{S}. Assume that 
\begin{equation}
\label{Bound-s}
0<\sup_{t\geq0}\|u(t)\|_{H^s}:=E<\infty.
\end{equation}
Then, there exist $R,\varepsilon>0$ depending on {$E,N,p,b,s$} such that if
\begin{equation}\label{crtr} 
\liminf_{t\to\infty}\int_{|x|<R}|u(t,x)|^2\,dx<\varepsilon^2,
\end{equation}
then, $u$ scatters for positive time. 
\end{prop}
Before proving Proposition \ref{crt}, let us first give a technical result.
\begin{lem}\label{tch}
Let $I$ be a time slab. There exists $\theta>0$ small enough such that the global solution $u$ to \eqref{S} satisfies
\begin{enumerate}
\item[1)]
$\|u-e^{-it\bD^{2s}}u_0\|_{L^a(I,L^r)}\lesssim \|u\|_{L^\infty(I,H^s)}^\theta\|u\|^{p-\theta}_{L^a(I,L^r)}$;
\item[2)]
$\|(1+|\nabla|^s)(u-e^{-it\bD^{2s}}u_0)\|_{L^q(I,L^r)}\lesssim \|u\|_{L^\infty(I,H^s)}^\theta\|u\|^{p-1-\theta}_{L^a(I,L^r)}\||\nabla|^su\|_{L^q(I,L^r)}$.
\end{enumerate}
\end{lem}
\begin{proof}[{Proof of Lemma \ref{tch}}]
~\begin{enumerate}
\item[1)]
Using H\"older's inequality, one writes
\begin{eqnarray*}
\||x|^{-b}|u|^{p-1}u\|_{L^{r'}(|x|<1)}=\||x|^{-b}|u|^{p-\theta}|u|^{\theta}\|_{L^{r'}(|x|<1)}
&\leq&\||x|^{-b}\|_{L^\mu(|x|<1)}\|u\|_{L^{\frac{2N}{N-2s}}}^\theta\|u\|_{L^r}^{p-\theta},
\end{eqnarray*}
where
$$
\frac{1}{r'}=\frac{1}{\mu}+\frac{\theta(N-2s)}{2N}+\frac{p-\theta}{r},
$$
provided that
$$
0<\theta<\frac{2N}{N-2s},\;0<\theta<p,\;p-r<\theta,\;\frac{b}{N}<\frac{1}{\mu}.
$$
The last integrability condition reads
\begin{eqnarray}
\frac bN
&<&\frac1\mu=1-\frac{\theta(N-2s)}{2N}-\frac{1+p-\theta}r\label{intg1}\\
&=&1-\frac{\theta(N-2s)}{2N}-\frac{(N-2s_c)(1+p-\theta)-4(s-s_c)}{2N}\nonumber\\
&=&\frac{2N-\theta(N-2s)-(N-2s_c)(1+p-\theta)+4(s-s_c)}{2N}.\nonumber
\end{eqnarray}
This is equivalent to
\begin{eqnarray*}
2b
&<&2N-\theta(N-2s)-(N-2s_c)(1+p-\theta)+4(s-s_c)\\
&=&2N-\theta(N-2s)-(N-2s_c)(-1+p+2-\theta)+4(s-s_c)\\
&=&2\theta(s-s_c)+2b.
\end{eqnarray*}
This is obviously satisfied and gives via Sobolev embedding 
\begin{eqnarray*}
\||x|^{-b}|u|^{p-1}u\|_{L^{d'}(I,L^{r'}(|x|<1))}
&\lesssim&\|u\|_{L^\infty(H^s)}^\theta\|\|u(t)\|_{L^r}^{p-\theta}\|_{L^{d'}(I)}\\
&\lesssim&\|u\|_{L^\infty(H^s)}^\theta\|u(t)\|^{p-\theta}_{L^{a}(I, L^r)}.
\end{eqnarray*}
Let us estimate the same term on the complementary to the unit ball. Using H\"older's inequality, we have
\begin{eqnarray*}
\||x|^{-b}|u|^{p-1}u\|_{L^{r'}(|x|>1)}
&\leq&\||x|^{-b}\|_{L^{\mu_1}(|x|>1)}\|u\|_{L^{r_1}}^\theta\|u\|_{L^r}^{p-\theta}.
\end{eqnarray*}
Here, the integrability condition reads
\begin{eqnarray}
\frac bN
&>&\frac1{\mu_1}=1-\frac{\theta}{r_1}-\frac{1+p-\theta}r\label{intg2}\\
&=&1-\frac{\theta}{r_1}-\frac{(N-2s_c)(1+p-\theta)-4(s-s_c)}{2N}.\nonumber
\end{eqnarray}
This is equivalent to
\begin{eqnarray*}
\frac{\theta}{r_1}
&>&1-\frac bN-\frac{(N-2s_c)(1+p-\theta)-4(s-s_c)}{2N}\\
&=&\frac{2N-2b-(N-2s_c)(-1+p+2-\theta)+4(s-s_c)}{2N}\\
&=&\frac{\theta(N-2s_c)}{2N}.
\end{eqnarray*}
Thus, it is sufficient to choose $r_1\in[2,\frac{2N}{N-2s_c})$. The first point follows with Strichartz estimates arguing as previously.
\item[2)]
Using the first point with the equality $\frac1{q'}=\frac{p-1-\theta}a+\frac1q$, one has by Strichartz estimates
\begin{eqnarray*}
\|u-e^{-it\bD^{2s}}u_0\|_{L^q(I,L^r)}
&\lesssim&\||x|^{-b}|u|^{p-1}u\|_{L^{q'}(I,L^{r'})}\\
&\lesssim&\|u\|_{L^\infty(I,H^s)}^\theta\|\|u\|^{p-1-\theta}_{L^r}\|u\|_{L^r}\|_{L^{q'}(I)}\\
&\lesssim&\|u\|_{L^\infty(I,H^s)}^\theta\|u\|^{p-1-\theta}_{L^a(I,L^r)}\|u\|_{L^q(I,L^r)}.
\end{eqnarray*}
Now, let us estimate the term
\begin{eqnarray*}
&&\|\bD^{s}[|x|^{-b}|u|^{p-1}u]\|_{L^{q'}(I,L^{r'})}\\
&\lesssim&\||x|^{-b-s}|u|^{p-1}u]\|_{L^{q'}(I,L^{r'}(|x|<1))}+\||x|^{-b-s}|u|^{p-1}u]\|_{L^{q'}(I,L^{r'}(|x|>1))}\\
&+&\||x|^{-b}\bD^{s}[|u|^{p-1}u]\|_{L^{q'}(I,L^{r'}(|x|<1))}+\||x|^{-b}\bD^{s}[|u|^{p-1}u]\|_{L^{q'}(I,L^{r'}(|x|>1))}.
\end{eqnarray*}
Using H\"older's inequality, one writes
\begin{eqnarray*}
\||x|^{-b-s}|u|^{p-1}u]\|_{L^{r'}(|x|<1)}
&\leq&\||x|^{-b-s}\|_{L^\mu(|x|<1)}\|\|u\|_{L^{\frac{2N}{N-2s}}}^\theta\|u\|_{L^r}^{p-1-\theta}\|u\|_{L^{\frac{rN}{N-rs}}}\\
&\lesssim&\|u\|_{H^s}^\theta\|u\|_{L^r}^{p-1-\theta}\|\bD^{s}u\|_{L^r}.
\end{eqnarray*}
Here, one has the same condition \eqref{intg1} above
$$\frac{s+b}N<\frac1\mu=1+\frac sN+\theta(-\frac12+\frac sN)-\frac{p+1-\theta}r.$$
On the complementary of the unit ball, one has
\begin{eqnarray*}
\||x|^{-b-s}|u|^{p-1}u]\|_{L^{r'}(|x|>1)}
&\leq&\||x|^{-b-s}\|_{L^{\mu_1}(|x|>1)}\|u\|_{L^{r_1}}^\theta\|u\|_{L^r}^{p-1-\theta}\|u\|_{L^{\frac{rN}{N-rs}}}\\
&\lesssim&\|u\|_{H^s}^\theta\|u\|_{L^r}^{p-1-\theta}\|\bD^{s}u\|_{L^r}.
\end{eqnarray*}
Here, one has the same condition \eqref{intg2} above
$$\frac{s+b}N>\frac1{\mu_1}=1+\frac sN+\frac\theta{r_1}-\frac{p+1-\theta}r.$$
Thus, one can choose $r_2\in[2,\frac{2N}{N-2s_c})$. Now, using the fractional chain rule in Lemma \ref{chain}, one gets
\begin{eqnarray*}
\||x|^{-b}\bD^{s}[|u|^{p-1}u]\|_{L^{r'}(|x|<1)}
&\lesssim&\||x|^{-b}\|_{L^{\mu_2}(|x|<1)}\|u\|_{L^{r_2}}^\theta\|u\|_{L^r}^{p-1-\theta}\|\bD^{s}u\|_{L^r}\\
&\lesssim&\|u\|_{H^s}^\theta\|u\|_{L^r}^{p-1-\theta}\|\bD^{s}u\|_{L^r}.
\end{eqnarray*}
Here
$$
\left\{
\begin{array}{ll}
\frac{b}N<\frac1{\mu_2}=1-\frac\theta{r_2}-\frac{p+1-\theta}r;\\
2\leq r_2\leq\frac{2N}{N-2s}
\end{array}
\right.
$$
This condition is satisfied as before for $r_2=\frac{2N}{N-2s}$. The last term reads
\begin{eqnarray*}
\||x|^{-b}\bD^{s}[|u|^{p-1}u]\|_{L^{r'}(|x|>1)}
&\lesssim&\||x|^{-b}\|_{L^{\mu_3}(|x|>1)}\|u\|_{L^{r_3}}^\theta\|u\|_{L^r}^{p-1-\theta}\|\bD^{s}u\|_{L^r}\\
&\lesssim&\|u\|_{H^s}^\theta\|u\|_{L^r}^{p-1-\theta}\|\bD^{s}u\|_{L^r}.
\end{eqnarray*}
Here
$$
\left\{
\begin{array}{ll}
\frac{b}N>\frac1{\mu_3}=1-\frac\theta{r_3}-\frac{p+1-\theta}r;\\
2\leq r_3\leq\frac{2N}{N-2s}
\end{array}
\right.
$$
This condition is satisfied as before for $r_3\in[2,\frac{2N}{N-2s_c})$. Now, plugging all the above estimates together and using the equality $\frac1{q'}=\frac{p-1-\theta}a+\frac1q$, we obtain
\begin{eqnarray*}
\|\bD^{s}[|x|^{-b}|u|^{p-1}u]\|_{L^{q'}(I,L^{r'})}
&\lesssim&\|\|u\|_{H^s}^\theta\|u\|_{L^r}^{p-1-\theta}\|\bD^{s}u\|_{L^r}\|_{L^{q'}(I)}\\
&\lesssim&\|u\|_{H^s}^\theta\|u\|_{L^a(I,L^r)}^{p-1-\theta}\|\bD^{s}u\|_{L^{q}(I,L^r)}.
\end{eqnarray*}
The proof is finished via Strichartz estimates.
\end{enumerate}
\end{proof}

The key of the proof of the scattering criterion is the next result.
\begin{prop}\label{fn}
Suppose that the assumptions of Proposition \ref{crt} are fulfilled. Then, for any $\varepsilon>0$, there exist $T,\mu>0$ satisfying 
\begin{equation}
\label{cst}
\|e^{-i(t-T)\bD^{2s}}u(T)\|_{L^a((T,\infty),L^r)}\lesssim \varepsilon^\mu.
\end{equation}
\end{prop}
\begin{proof}
 By the integral formula
\begin{eqnarray}
\nonumber
e^{-i(t-T)\bD^{2s}}u(T)
&=&e^{-it\bD^{2s}}u_0+i\int_0^Te^{-i(t-\tau)\bD^{2s}}[|x|^{-b}|u|^{p-1}u]\,d\tau\\
\nonumber
&=&e^{-it\bD^{2s}}u_0+i\Big(\int_0^{T-\varepsilon^{-\beta}}+\int_{T-\varepsilon^{-\beta}}^T\Big)e^{-i(t-\tau)\bD^{2s}}[|x|^{-b}|u|^{p-1}u]\,d\tau\\
\label{F1-2}
&:=&e^{-it\bD^{2s}}u_0+F_1+F_2.
\end{eqnarray}
\begin{enumerate}
\item[$\bullet$]  {\bf Estimate of the linear term}.\newline Since $(a,\frac{rN}{N+rs_c})\in\Gamma$, by Strichartz estimate and Sobolev embedding, one has
\begin{eqnarray*}
\|e^{-i\cdot\bD^{2s}}u_0\|_{L^a((T,\infty),L^r)}
&\lesssim&\||\nabla|^{s_c}e^{-i\cdot\bD^{2s}}u_0\|_{L^a((T,\infty),L^\frac{rN}{N+rs_c})}\lesssim\|u_0\|_{H^s}.
\end{eqnarray*}
Thus, one may choose $T_0>\varepsilon^{-\beta}>0$, where $\beta>0$, such that 
\begin{equation}
\label{cst1}
\|e^{-i\cdot\bD^{2s}}u_0\|_{L^a((T_0,\infty),L^r)}\leq \varepsilon^2.
\end{equation}
\item[$\bullet$] {\bf Estimate of the term $F_2$}.\newline By the assumption \eqref{crtr}, one has for $T>\varepsilon^{-\beta}$ large enough,
$$\int_{\R^N}\psi_R(x)|u(T,x)|^2\,dx<\varepsilon^2.$$
Moreover, a computation with the use of \eqref{S} gives
\begin{eqnarray*}
\frac d{dt}\int_{\R^N}\psi_R(x)|u(t,x)|^2\,dx
&=&2\int_{\R^N}\psi_R(x)\Re[\dot u(t,x)\bar u(t,x)]\,dx\\
&=&2\int_{\R^N}\psi_R(x)\Im[\bD^{2s}u(t,x)\bar u(t,x)]\,dx.
\end{eqnarray*}
Hence, 
\begin{eqnarray*}
\frac d{dt}\int_{\R^N}\psi_R|u|^2\,dx
&=&2\Im\int_{\R^N}\bD^{s}(\psi_R\bar u)\bD^{s}u\,dx\\
&=&2\Im\int_{\R^N}\bar u\bD^{s}\psi_R\bD^{s}u\,dx\\
&+&2\Im\int_{\R^N}\bD^{s}u\Big(\bD^{s}(\psi_R\bar{u})-\bar{u}\bD^{s}(\psi_R)-\psi_R\bD^{s}(\bar{u})\Big)\,dx.
\end{eqnarray*}
{By Lemma \ref{lbnz} with $\psi_R$ rather than $v$, one has via Sobolev embedding}{
\begin{eqnarray*}
\bigg|\frac d{dt}\int_{\R^N}\psi_R(x)|u(t,x)|^2\,dx\bigg|&\lesssim& R^{-s}+\|\bD^{s_1}\psi_R\|_{L^{p_1}}\,\|\bD^{s_2}\bar{u}\|_{L^{p_2}},\\
&\lesssim& R^{-s}+R^{\frac{N}{p_1}-s_1}\|u\|_{H^s},
\end{eqnarray*}
provided that $s_1, s_2, p_1, p_2$ satisfy
$$
s=s_1+s_2,\quad2\leq p_2\leq \frac{2N}{N-2(s-s_2)},\quad \frac{1}{2}=\frac{1}{p_1}+\frac{1}{p_2}.
$$
Let us choose $s_1=\theta\,s$ and $p_1=\frac{2N}{\theta\,s}$ for some $\theta\in(0,1)$. Hence 
$p_2=\frac{2N}{N-\theta\,s}$ and $\frac{N}{p_1}-s_1=-\frac{\theta}{2}s$. It follows that
\begin{eqnarray*}
\bigg|\frac d{dt}\int_{\R^N}\psi_R(x)|u(t,x)|^2\,dx\bigg|
&\lesssim& R^{-s}+R^{-\frac{\theta}{2}s}\|u\|_{H^s},\\
&\leq&C(N,s,\psi,\theta)\,R^{-\frac{\theta}{2}s},
\end{eqnarray*}
where $C(N,s,\psi,\theta)$ comes from the sharp constant in the Sobolev embedding \cite[Theorem 1.1, p. 226]{BCS} and $\|\mathbf{D}^{s_1}\psi\|_{p_1}$. Since 
$$
C(N,s,\psi,\theta)\to C(N,s,\psi)>0 \;\;\text{as}\;\;\theta\to 1,
$$
we deduce that
$$
\bigg|\frac d{dt}\int_{\R^N}\psi_R(x)|u(t,x)|^2\,dx\bigg|\lesssim R^{-\frac{s}{2}}.
$$
}
Therefore, for any $T-\varepsilon^{-\beta}\leq t\leq T$ and $R>\varepsilon^{-\frac{2(2+\beta)}{s}}$, we obtain
$$\|\psi_Ru(t)\|\leq\Big( \int_{\R^N}\psi_R(x)|u(T,x)|^2\,dx+C\frac{T-t}{R^\frac s2}\Big)^\frac12\leq C\varepsilon.$$
This gives
$$\|\psi_Ru\|_{L^\infty([T-\varepsilon^{-\beta},T],L^2)}\leq C\varepsilon.$$
By Strichartz estimate and Lemma \ref{tch}, one has
\begin{eqnarray*}
\|\psi_R|x|^{-b}|u|^p\|_{L^{d'}((T-\varepsilon^{-\beta},T),L^{r'})}
&\lesssim&\|u\|_{L^\infty((T-\varepsilon^{-\beta},T),H^1)}^\theta\|u\|_{L^a(J_1,L^r)}^{p-1-\theta}\|\psi_Ru\|_{L^a((T-\varepsilon^{-\beta},T),L^r)}\\
&\lesssim&|(T-\varepsilon^{-\beta},T)|^{\frac{p-1-\theta}{a}}\|u\|_{L^\infty((T-\varepsilon^{-\beta},T),H^1)}^{p-1}\|\psi_Ru\|_{L^a((T-\varepsilon^{-\beta},T),L^r)}\\
&\lesssim&\varepsilon^{-\beta\frac{p-1-\theta}{a}}\|\psi_Ru\|_{L^\infty((T-\varepsilon^{-\beta},T),L^r)}.
\end{eqnarray*}
Using Gagliardo-Nirenberg inequality, one has
\begin{eqnarray*}
\|\psi_R|x|^{-b}|u|^p\|_{L^{d'}((T-\varepsilon^{-\beta},T),L^{r'})}
&\lesssim&\varepsilon^{-\beta\frac{p-1-\theta}{a}}\|\psi_Ru\|_{L^\infty(J_1,L^2)}^{1-\frac Ns(\frac12-\frac1r)}\\
&\lesssim&\varepsilon^{-\beta\frac{p-1-\theta}{a}}\varepsilon^{1-\frac Ns(\frac12-\frac1r)}\\
&\lesssim&\varepsilon^{(1-\beta)\frac{(s-s_c)(p-1-\theta)}{s(1+p-\theta)}}.
\end{eqnarray*}
Now, let us estimate 
\begin{eqnarray*}
\|(1-\psi_R)|x|^{-b}|u|^p\|_{L^{r'}}
&\lesssim&\||x|^{-b}u^{p-\theta}u^\theta\|_{L^{r'}(|x|>R/2)}\\
&\lesssim&\||x|^{-b}\|_{L^{\mu}(|x|>R/2)}\|u\|_{L^{r_1}}^\theta\|u\|^{p-\theta}_{L^r}\\
&\lesssim&R^{N-b\mu}\|u\|_{L^{r_1}}^\theta\|u\|^{p-\theta}_{L^r}.
\end{eqnarray*}
Here
\begin{align*}
b>\frac N\mu&=N-\frac{N\theta}{r_1}-\frac{N(1+p-\theta)}r\\
&=N-\frac{N\theta}{r_1}-(\frac N2-s_c)(1+p-\theta)+2(s-s_c).
\end{align*}
This is equivalent to
\begin{align*}
r_1&<\frac{N\theta}{-b+N-(\frac N2-s_c)(1+p-\theta)+2(s-s_c)}\\
&=\frac{N\theta}{-b+N-\frac{2s-b}{p-1}(p-1+2-\theta)+2(s-\frac N2+\frac{2s-b}{p-1})}\\
&=\frac{N(p-1)}{2s-b}.
\end{align*}
Since $p>p_c$, one can choose $2\leq r_1\leq 2^*$. Thus, by Sobolev embedding
$$
\|(1-\psi_R)|x|^{-b}|u|^p\|_{L^{r'}}\lesssim R^{N-b\mu}\|u\|_{H^1}^\theta\|u\|^{p-\theta}_{L^r}\lesssim R^{N-b\mu}\|u\|^{p-\theta}_{H^1}.
$$
Hence,
\begin{eqnarray*}
\|(1-\psi_R)|x|^{-b}|u|^p\|_{L^{d'}((T-\varepsilon^{-\beta},T),L^{r'})}
&\lesssim&R^{N-b\mu}\varepsilon^{-\frac\beta{d'}}\\
&\lesssim&R^{N-b\mu}\varepsilon^{-\frac{\beta(s-s_c)(p-\theta)}{s(1+p-\theta)}}\\
&\lesssim&\varepsilon^{|N-b\mu|\frac{2+\beta}s-\frac{\beta(s-s_c)(p-\theta)}{s(1+p-\theta)}}.
\end{eqnarray*}
Then, choosing $0<\beta<<1$, there is $\gamma>0$ such that
\begin{eqnarray}
\nonumber
\|F_2\|_{L^a((T,\infty),L^r)}
&\lesssim&\varepsilon^{(1-\beta)\frac{(s-s_c)(p-1-\theta)}{s(1+p-\theta)}}+\varepsilon^{|N-b\mu|\frac{2+\beta}s-\frac{\beta(s-s_c)(p-\theta)}{s(1+p-\theta)}}\\
\label{cst2}
&\lesssim&\varepsilon^\gamma.
\end{eqnarray}

\item[$\bullet$] {\bf Estimate of the term $F_1$}.\newline Take $\frac1c:=\frac1r+\frac{s_c}N$. Then $(a,c)\in\Gamma$ and there is $\lambda\in[0,1]$ such that $\frac1r:=\frac\lambda c$. By interpolation via the mass conservation, one writes
\begin{eqnarray*}
\|F_1\|_{L^a((T,\infty),L^r)}
&\lesssim&\|F_1\|_{L^a((T,\infty),L^c)}^\lambda\|F_1\|_{L^a((T,\infty),L^\infty)}^{1-\lambda}\\
&\lesssim&\|e^{-i(t-(T-\varepsilon^{-\beta}))\bD^{2s}}u(T-\varepsilon^{-\beta})-e^{-it\bD^{2s}}u_0\|_{L^a((T,\infty),L^c)}^\lambda\|F_1\|_{L^a((T,\infty),L^\infty)}^{1-\lambda}\\
&\lesssim&\|F_1\|_{L^a((T,\infty),L^\infty)}^{1-\lambda}.
\end{eqnarray*}
Using the dispersive estimate \eqref{free}, one has for $T\leq t$,
\begin{eqnarray*}
\|F_1\|_{L^{\infty}}
&\lesssim&\int_0^{T-\varepsilon^{-\beta}}\frac1{(t-s)^{\frac N2}}\|\bD^{N(1-s)}\Big(|x|^{-b}|u|^{p-1}u\Big)\|_{L^{1}}\,ds.
\end{eqnarray*}
Thanks to the fractional chain rule Lemma \ref{chain}, we get
\begin{eqnarray}
\|\bD^{N(1-s)}\Big(|x|^{-b}|u|^{p-1}u\Big)\|_{L^{1}}
&\lesssim&\||x|^{-b-N(1-s)}|u|^p\|_{L^{1}}+\||x|^{-b}|u|^{p-1}\bD^{N(1-s)}u\|_{L^{1}}\label{rstr}\\
&:=&(I)+(II).
\end{eqnarray}
By H\"older estimate, the term $(I_1):=\||x|^{-b-N(1-s)}|u|^p\|_{L^1(|x|<1)}$ satisfies
\begin{eqnarray*}
(I_1)
&\leq&\||x|^{-b-N(1-s)}\|_{L^\mu(|x|<1)}\|u\|_{L^{d}}^p\lesssim\|u\|_{H^s}^p.
\end{eqnarray*}
Here, $\mu$ and $d$ satisfy
$$
\left\{
\begin{array}{ll}
1=\frac1\mu+\frac pd,\\
\mu<\frac N{b+N(1-s)},\\
2\leq d\leq\frac{2N}{N-2s}.
\end{array}
\right.
$$
This requires that
$$
d\in(\frac{Np}{sN-b},\infty)\cap(2,\frac{2N}{N-2s}),
$$
which is possible if $p<2\frac{Ns-b}{N-2s}$. Moreover, the condition $p^*\leq 2\frac{Ns-b}{N-2s}$ is equivalent to $s\geq\frac N{2(N-1)}$, which is satisfied because $s<1$ and $N\geq2$.\\
By H\"older's inequality, the term $(I_2):=\||x|^{-b-N(1-s)}|u|^p\|_{L^1(|x|>1)}$ can be estimated as
\begin{eqnarray*}
(I_2)
&\leq&\||x|^{-b-N(1-s)}\|_{L^\gamma(|x|>1)}\|u\|_{L^{d}}^e\lesssim\|u\|_{H^s}^p.
\end{eqnarray*}
Here, $\gamma$ and $e$ satisfy
$$
\left\{
\begin{array}{ll}
1=\frac1\gamma+\frac pe,\\
\gamma>\frac N{b+N(1-s)},\\
2\leq e\leq\frac{2N}{N-2s}.
\end{array}
\right.
$$
This requires that
$$e\in(1,\frac{pN}{Ns-b})\cap(2,\frac{2N}{N-2s}),$$
which is possible if and only if $p>\frac{2(Ns-b)}N$.\\
Again, by H\"older's inequality, the term $(II_1):=\||x|^{-b}|u|^{p-1}\bD^{N(1-s)}u\|_{L^1(|x|<1)}$ satisfies
\begin{eqnarray*}
(II_1)
&\leq&\||x|^{-b}\|_{L^\beta(|x|<1)}\|u\|_{L^{f}}^{p-1}\|\bD^{N(1-s)}u\|\\
&\lesssim&\|u\|_{H^s}^{p-1}\|\bD^{N(1-s)}u\|\\
&\lesssim&\|u\|_{H^s}^{p-1}\|u\|^\theta\|\bD^{s}u\|^{1-\theta}\\
&\lesssim&\|u\|_{H^s}^{p}.
\end{eqnarray*}
Here, $\theta\in[0,1]$ and $\beta$ and $f$ satisfy
$$
\left\{
\begin{array}{ll}
1=\frac1\beta+\frac{p-1}f+\frac12,\\
\beta<\frac N{b},\\
2\leq f\leq\frac{2N}{N-2s},\\
0\leq N(1-s)\leq s.
\end{array}
\right.
$$
This is equivalent to
$$
\left\{
\begin{array}{ll}
\frac12-\frac{p-1}f=\frac1\beta>\frac{b}N,\\
2\leq f\leq\frac{2N}{N-2s},\\
s\geq\frac N{1+N},
\end{array}
\right.
$$
which is possible if $\frac{N-2s}{2N}\leq\frac1f<\frac{1}{p-1}(\frac{1}{2}-\frac{b}{N})$. Thus,
\begin{equation}
\label{upper-p}
p-1<\frac{2N}{N-2s}\left(\frac{1}{2}-\frac{b}{N} \right).
\end{equation}

Using again H\"older's inequality, we estimate $(II_2):=\||x|^{-b}|u|^{p-1}\bD^{N(1-s)}u\|_{L^1(|x|>1)}$ as
\begin{eqnarray*}
(II_2)
&\leq&\||x|^{-b}\|_{L^\alpha(|x|>1)}\|u\|_{L^{j}}^{p-1}\|\bD^{N(1-s)}u\|\\
&\lesssim&\|u\|_{H^s}^{p-1}\|\bD^{N(1-s)}u\|\\
&\lesssim&\|u\|_{H^s}^{p-1}\|u\|^\theta\|\bD^{s}u\|^{1-\theta}\\
&\lesssim&\|u\|_{H^s}^p.
\end{eqnarray*}
Here, $\theta\in[0,1]$, $\alpha$ and $j$ satisfy
$$
\left\{
\begin{array}{ll}
1=\frac1\alpha+\frac{p-1}j+\frac12,\\
\alpha>\frac N{b},\\
2\leq j\leq\frac{2N}{N-2s},\\
s\geq\frac N{1+N}.
\end{array}
\right.
$$
This leads to
$$
\left\{
\begin{array}{ll}
\frac12-\frac bN<\frac{p-1}j\leq\frac{p-1}2,\\
2\leq j\leq\frac{2N}{N-2s},\\
s\geq\frac N{1+N},
\end{array}
\right.
$$
which is possible if 
\begin{equation}
\label{lower-p}
p-1>1-\frac{2b}N.
\end{equation}

Summarize the above estimates, we infer that $(II)\lesssim\|u\|_{H^s}^{p}$ if 
\begin{equation}
\label{low-upp-p}
2-\frac{2b}N<p<1+\frac{2N}{N-2s}(\frac12-\frac bN).
\end{equation}

Under the above restrictions \eqref{low-upp-p}, one has
\begin{eqnarray*}
\|F_1\|_{L^{\infty}}
&\lesssim&\int_0^{T-\varepsilon^{-\beta}}\frac1{(t-\tau)^{\frac N2}}\|\bD^{N(1-s)}\Big(|x|^{-b}|u|^{p-1}u\Big)\|_{L^{1}}\,d\tau\\
&\lesssim&\int_0^{T-\varepsilon^{-\beta}}\frac1{(t-\tau)^{\frac N2}}\|u(\tau)\|_{H^s}^{p}\,d\tau\\
&\lesssim&(t-T+\varepsilon^{-\beta})^{1-\frac N2}.
\end{eqnarray*}
It follows that, if $\frac N2-1-\frac1a>0$,
\begin{eqnarray}
\nonumber
\|F_1\|_{L^a((T,\infty),L^r)}
&\lesssim&\|F_1\|_{L^a((T,\infty),L^\infty)}^{1-\lambda}\\
\nonumber
&\lesssim&\Big(\int_T^\infty(t-T+\varepsilon^{-\beta})^{a[1-\frac N2]}\,dt\Big)^{\frac{1-\lambda}a}\\
\label{cst3}
&\lesssim&\varepsilon^{(1-\lambda)\beta(\frac N2-1-\frac1a)}.
\end{eqnarray}
Note that the condition $\frac N2-1-\frac1a>0$ translate to
\begin{equation}\label{false}
\frac N2-1>\frac{s-s_c}{s(1+p)}.
\end{equation}
Clearly \eqref{false} {is false for $N=2$} and reads for $N\geq3$, $Q(p-1)>0$, where $Q$ is the polynomial function 
$$Q(X):=s(N-2)X^2+((1+2s)N-6s)X-2(2s-b).$$
Let us compute 
\begin{eqnarray*}
Q(p^*-1)
&=&s(N-2)\Big(\frac{2(2s-b)}{N-2s}\Big)^2+((1+2s)N-6s)\Big(\frac{2(2s-b)}{N-2s}\Big)-2(2s-b)\\
&=&\frac{2(2s-b)}{(N-2s)^2}\Big[2s(N-2)(2s-b)+((1+2s)N-6s)(N-2s)-(N-2s)^2\Big]\\
&=&\frac{2(2s-b)}{(N-2s)^2}\Big[2s(N-2)(2s-b)+(N-2s)((1+2s)N-6s-(N-2s))\Big]\\
&=&\frac{2(2s-b)}{(N-2s)^2}\Big[2s(N-2)(2s-b)+2s(N-2s)^2\Big]\\
&=&\frac{4s(2s-b)}{(N-2s)^2}\Big[(N-2)(2s-b)+(N-2s)^2\Big]\\
&>&0.
\end{eqnarray*}
{Moreover, one can easily verify that 
$$
Q(p_*-1)=\frac{4s(2s-b)}{N^2}\Big[(N-2)(2s-b)+N(N-3)\Big]>0,
$$
provided that $N\geq 3$. This means that $Q(p-1)>0$ as long as $p\in(p_*, p^*)$.}

\end{enumerate}
Now, the desired estimate \eqref{cst} easily follows from \eqref{cst1}, \eqref{cst2} and \eqref{cst3}. This finishes the proof of Proposition \ref{fn}. 
\end{proof}
Now, we are ready to prove the scattering criterion.
\begin{proof}[{Proof of Proposition \ref{crt}}]
Take $0<\varepsilon<<1$. 
By Proposition \ref{fn}, there exists $\mu>0$ such that
$$\|e^{-it\bD^{2s}}u(T)\|_{L^a((0,\infty),L^r)}=\|e^{-i(t-T)\bD^{2s}}u(T)\|_{L^a((T,\infty),L^r)}\lesssim\varepsilon^\mu.$$
So, with Lemma \ref{tch} together with the continuity argument Lemma \ref{boots}, one gets
$$\|u\|_{L^a((T,\infty),L^r)}\lesssim\varepsilon^\mu.$$
Again, thanks to Lemma \ref{tch} via the continuity argument Lemma \ref{boots}, we obtain the global bound
$$u\in L^q(\R, W^{s,r}).$$
Now, for $t>t'>>1$,  we have by Lemma \ref{tch} and Strichartz estimates
\begin{eqnarray*}
\|u(t')-u(t)\|_{H^s}
&\lesssim&\|(1+|\nabla|^s)[|x|^{-b}|u|^{p-1}u]\|_{S(t',t)}\\
&\lesssim&\|u\|_{L^a((t,t'),L^r)}^{p-1}\|u\|_{L^q((t',t),W^{s,r})}\\
&&\to0.
\end{eqnarray*}

The proof is achieved via the Cauchy criterion.
\end{proof}
\section{Proof of  Theorem \ref{t1}}
\label{S6}
Let $R,\varepsilon>0$ given by Proposition \ref{crt} and $R_n\to\infty$. Let $t_n\to\infty$ given by Corollary \ref{bnd2}. For $n>>1$ such that $R_n>R$, one gets by H\"older's inequality
\begin{eqnarray*}
\int_{|x|\leq R}|u(t_n,x)|^2\,dx
&\leq &R^{\frac{2b}{1+p}}\int_{|x|\leq R}|x|^{-\frac{2b}{1+p}}|u(t_n,x)|^2\,dx\\
&\leq&R^{\frac{2b}{1+p}}\left|B(0,R)\right|^{\frac{1+p}{p-1}}\,\||x|^{-\frac{2b}{1+p}}|u(t_n,x)|^2\|_{L^\frac{1+p}2(|x|\leq R_n)}\\
&\leq&R^{\frac{2b+N(p-1)}{1+p}}\left(\int_{|x|\leq R_n}|x|^{-b}|u(t_n,x)|^{1+p}\,dx\right)^\frac2{1+p}\to0.
\end{eqnarray*}
Hence, the part (i) of Theorem \ref{t1}  follows from Proposition \ref{crt}.

Let us turn now to part (ii) of Theorem \ref{t1}. We have
\begin{eqnarray*}
\frac d{dt}M_R[u(t)]
&=&4\int_0^\infty m^s\int_{\R^N}|\nabla u_m|^2\,dx\,dm-4\int_0^\infty m^s\int_{R<|x|<2R}\left(1-{f''}\left(\frac{|x|}{R}\right)\right)|\nabla u_m|^2\,dx\,dm\\
&-&\int_0^\infty m^s\int_{\R^N}\Delta^2{f_R}|u_m|^2\,dx\,dm\\
&-&\frac{4sB}{1+p}\int_{|x|<R}|x|^{-b}|u|^{p+1}\,dx+
\frac{2(p-1)}{1+p}\int_{|x|>R}(N-\Delta{f}_R)|x|^{-b}|u|^{p+1}\,dx\\
&-&\frac{4b}{1+p}\int_{|x|>R}\frac{x\cdot\nabla{f}_R}{|x|^2}|x|^{-b}|u|^{1+p}\,dx\\
&\leq&4s\|\bD^{s} u\|^2+CR^{-2s}\\
&-&\frac{4sB}{1+p}\int_{|x|<R}|x|^{-b}|u|^{p+1}\,dx+\frac{2(p-1)}{1+p}\int_{|x|>R}(N-\Delta{f}_R)|x|^{-b}|u|^{p+1}\,dx\\
&-&\frac{4b}{1+p}\int_{|x|>R}\frac{x\cdot\nabla{f}_R}{|x|^2}|x|^{-b}|u|^{1+p}\,dx,
\end{eqnarray*}
where $M_R[u(t)]$ is given by \eqref{M-R}.
Take $0<\varepsilon<\frac{p-1}{s}(s-\frac12)$ and $\frac12<\alpha:=\frac12+\frac{\varepsilon s}{p-1}<s<\frac N2$. Using the radial assumption and the Strauss type estimate in Lemma \ref{F-Sob-I}, via the interpolation $\|\bD^{\alpha}\phi\|\lesssim\|\phi\|^{1-\frac\alpha s}\,\|\bD^{s}\phi\|^{\frac\alpha s} $, one gets
\begin{eqnarray*}
\int_{|x|>R}|x|^{-b}|u(t,x)|^{1+p}\,dx
&\leq&R^{-b}\|u(t)\|^2\|u(t)\|_{L^\infty(|x|>R)}^{p-1}\\
&\lesssim&R^{-b}\Big(R^{\alpha-\frac N2}\|\bD^{\alpha}u(t)\|\Big)^{p-1}\\
&\lesssim&R^{-b-\frac{(N-1)(p-1)}2+\varepsilon s}\|\bD^{s}u(t)\|^{(p-1)\frac\alpha s}\\
&\lesssim&R^{-b-\frac{(N-1)(p-1)}2+\varepsilon s}\|\bD^{s}u(t)\|^{\varepsilon+\frac{p-1}{2s}}.
\end{eqnarray*}
Therefore,
\begin{eqnarray*}
\frac{d}{dt}M_R[u(t)]
&\leq&4sI[u(t)]+CR^{-2s}+CR^{-b-\frac{p-1}2(N-1)+\varepsilon s}\|\bD^{s} u(t)\|^{\varepsilon+\frac{p-1}{2s}}.
\end{eqnarray*}
If $u$ is global and does not blow-up in infinite time, it follows that $\|\bD^{s} u\|_{L^\infty(\R,L^2)}\lesssim1$.
Thus, taking $R>>1$ and using \eqref{bl}, one gets
$\frac d{dt}M_R[u(t)]\leq-2\eta<0$, for some $\eta>0$. Hence, for $t>0$ large enough, we get
\begin{equation}
\label{MR-bound}
M_R[u(t)]\leq-\eta t.
\end{equation}

Combining \eqref{MR-bound} with {\cite[Lemma A.1]{bhl}}, we obtain, for $t$ large enough,
\begin{equation}
\label{MR-bound1}
\eta t\leq -M_R[u(t)]=|M_R[u(t)]|{\lesssim \Big(1+\|\bD^{s}u(t)\|^{\frac{1}{s}}\Big).}
\end{equation}
This gives $\|\bD^{s}u(t)\|\gtrsim t^s$ for $t$ large enough, which is a contradiction.
This finishes the proof of part (ii) of Theorem \ref{t1}.  
\section{Proof of Corollary \ref{t2}}
\label{S7}


The scattering part follows by Theorem \ref{t1} with the next result. Indeed, the classical scattering condition below the ground state threshold is stronger than \eqref{ss1}.
\begin{lem}
\label{L71}
Suppose that assumptions \eqref{t11} and \eqref{t12} are fulfilled. Then, there exists $\varepsilon>0$ such that \eqref{1} is satisfied. 
\end{lem}
\begin{proof}
Using the identity $A+2\gamma_c=B\gamma_c$, we have
\begin{eqnarray*}
E[u][M[u]]^{\gamma_c}
&\geq&\|\nabla u\|^2\|u\|^{2\gamma_c}-\frac{2K_{opt}}{1+p}\|u\|^{A+2\gamma_c}\|\nabla u\|^B\\
&=&g(\|\nabla u\|\|u\|^{\gamma_c}),
\end{eqnarray*}
where $g(\tau)=\tau^2-\frac{2K_{opt}}{1+p}\tau^{B}$.

Now, with Pohozaev identities and the conservation laws, one has for some $0<\varepsilon<1$, 
\begin{eqnarray*}
g(\|\nabla u\|\|u\|^{\gamma_c})
&\leq&E[u][M[u]]^{\gamma_c}\\
&<&(1-\varepsilon)E[Q][M[Q]]^{\gamma_c}\\
&=&(1-\varepsilon)g(\|\nabla\,Q\|\|Q\|^{\gamma_c}).
\end{eqnarray*}
Thus, with time continuity, \eqref{t12} is invariant under the flow \eqref{S} and hence $T^*=\infty$. Moreover, by Pohozaev identities, one writes
$$E[Q][M[Q]]^{\gamma_c}=\frac{B-2}{B}(\|\nabla\,Q\|\|Q\|^{\gamma_c})^2=\frac{K_{opt}(B-2)}{1+p}(\|\nabla\,Q\|\|Q\|^{\gamma_c})^B$$
and so
$$1-\varepsilon\geq\frac B{B-2}\Big(\frac{\|\nabla u\|\|u\|^{\gamma_c}}{\|\nabla\,Q\|\|Q\|^{\gamma_c}}\Big)^2-\frac2{B-2}\Big(\frac{\|\nabla u\|\|u\|^{\gamma_c}}{\|\nabla\,Q\|\|Q\|^{\gamma_c}}\Big)^B.$$
Following the variations of $\tau\mapsto\frac B{B-2}\tau^2-\frac2{B-2}\tau^B$ via the assumption \eqref{t12} and a continuity argument, there is a real number denoted also by $0<\varepsilon<1$,  such that
\begin{equation}
\label{bbd1}
\|\nabla u\|\|u\|^{\gamma_c}\leq (1-\varepsilon)\|\nabla\,Q\|\|Q\|^{\gamma_c}\quad\mbox{on}\quad \R.
\end{equation}

Now, by \eqref{bbd1} and Pohozaev identities, ther exits  $0<\varepsilon<1$, such that
\begin{eqnarray*}
P[u][M[u]]^{\gamma_c}
&\leq& K_{opt}\|\nabla u\|^B\|u\|^{A+2\gamma_c}\\
&\leq&K_{opt}(1-\varepsilon)(\|\nabla\,Q\|\|Q\|^{\gamma_c})^B\\
&\leq&(1-\varepsilon)\frac{1+p}B(\|\nabla\,Q\|\|Q\|^{\gamma_c})^2\\
&\leq&(1-\varepsilon)P[Q]M[Q]^{\gamma_c}.
\end{eqnarray*}
This finishes the proof of Lemma \ref{L71}.
\end{proof}


Let us turn now to the blow-up part in Corollary \ref{t2}. Assume that \eqref{t13} is satisfied. Taking into account \cite[Section 3]{pz}, one has $\mathcal{MG}[u(t)]>1$ on $[0,T^*)$. Thus, by the Pohozaev identity $B\,E[Q]=(B-2)\|\bD^{s}\,Q\|^2$. It follows that
\begin{eqnarray*}
I[u][M[u]]^{\gamma_c}
&=&\Big(\|\bD^{s} u\|^2-\frac{B}{1+p}P[u]\Big)[M[u]]^{\gamma_c}\\
&=&\frac{B}{2}E[u][M[u]]^{\gamma_c}-(\frac{B}{2}-1)\|\bD^{s} u\|^2[M[u]]^{\gamma_c}\\
&\leq&\frac{B}{2}(1-\varepsilon)E[Q][M[Q]]^{\gamma_c}-(\frac{B}{2}-1)\|\bD^{s}Q\|^2[M[Q]]^{\gamma_c}\\
&\leq&-\varepsilon\|\bD^{s}\,Q\|^2[M[Q]]^{\gamma_c}.
\end{eqnarray*}
The proof follows by the use of Theorem \ref{t1}.

\begin{appendix}
\section{The two dimensional case}\label{appendix1}
One keeps the notations of section 5 about the scattering criteria and the estimate of the term $F_1$ for $N=2$. Take $\frac1c:=\frac1r+\frac{s_c}N$. Then $(a,c)\in\Gamma$ and for $f\geq r$, there is $\lambda\in[0,1]$ such that $$\frac1r:=\frac\lambda c+\frac{1-\lambda}f=\lambda(\frac1r+\frac{s_c}N)+\frac{1-\lambda}f.$$
Indeed, $\frac1r-\frac1f=\lambda(\frac1r+\frac{s_c}N-\frac1f)$ gives 
$$\frac{\frac1r-\frac1f}{\frac1r+\frac{s_c}N-\frac1f}=\lambda\in[0,1], \quad\mbox{for}\quad f\geq r.$$
 By interpolation via the mass conservation, one writes
\begin{eqnarray*}
\|F_1\|_{L^a((T,\infty),L^r)}
&\lesssim&\|F_1\|_{L^a((T,\infty),L^c)}^\lambda\|F_1\|_{L^a((T,\infty),L^f)}^{1-\lambda}\\
&\lesssim&\|e^{-i(t-(T-\varepsilon^{-\beta}))\bD^{2s}}u(T-\varepsilon^{-\beta})-e^{-it\bD^{2s}}u_0\|_{L^a((T,\infty),L^c)}^\lambda\|F_1\|_{L^a((T,\infty),L^f)}^{1-\lambda}\\
&\lesssim&\|F_1\|_{L^a((T,\infty),L^f)}^{1-\lambda}.
\end{eqnarray*}
With the free fractional Schr\"odinger operator dispersive estimate \eqref{free}, 
 one has for $T\leq t$ and $f:=\frac2\delta$, such that $1\geq\frac2r\geq\delta>0$,
\begin{eqnarray*}
\|F_1\|_{\frac2\delta}
&\lesssim&\int_0^{T-\varepsilon^{-\beta}}\frac1{(t-s)^{\frac{N(1-\delta)}2}}\|\bD^{N(1-s)(1-\delta)}\Big(|x|^{-b}|u|^{p-1}u\Big)\|_{L^{\frac2{2-\delta}}}\,ds.
\end{eqnarray*}
Using the fractional chain rule in Lemma \ref{chain}, one writes
\begin{eqnarray*}
\|\bD^{N(1-s)(1-\delta)}\Big(|x|^{-b}|u|^{p-1}u\Big)\|_{L^{\frac2{2-\delta}}}
&\lesssim&\||x|^{-b-N(1-s)(1-\delta)}|u|^p\|_{L^{\frac2{2-\delta}}}+\||x|^{-b}|u|^{p-1}\bD^{N(1-s)(1-\delta)}u\|_{L^{\frac2{2-\delta}}}\nonumber\\
&:=&(I)+(II).
\end{eqnarray*}
The first term $(I)$ can be estimated as previously. It remains to control the second term. By H\"older's inequality and Sobolev embedding, we have
\begin{eqnarray*}
(II)
&=&\||x|^{-b}|u|^{p-1}\bD^{N(1-s)(1-\delta)}u\|_{L^{\frac2{2-\delta}}}\\
&=&\|[|x|^{-s}u]^{\frac{b}s}u^{p-1-\frac{b}s}\bD^{N(1-s)(1-\delta)}u\|_{L^{\frac2{2-\delta}}}\\
&\leq&\||x|^{-s}u\|^{\frac{b}s}\|u\|_{L^{r_1}}^{p-1-\frac{b}s}\|\bD^{N(1-s)(1-\delta)}u\|_{L^{r_2}}\\
&\lesssim&\|u\|_{\dot H^s}^{\frac{b}s}\|u\|_{{L^{r_1}}}^{p-1-\frac{b}s}\|u\|_{ H^s}\\
&\lesssim&\|u\|_{H^s}^{1+\frac{b}s}\|u\|_{{L^{r_1}}}^{p-1-\frac{b}s}\\
&\lesssim&\|u\|_{H^s}^p.
\end{eqnarray*}
Here, on takes
$$
\left\{
\begin{array}{ll}
\frac{2-\delta}2=\frac{b}{2s}+\frac{p-1-\frac{b}s}{r_1}+\frac1{2},\\
p\geq1+\frac{b}s,\\
0\leq N(1-s)(1-\delta)\leq s,\\
2\leq r_1\leq\frac{2N}{N-2s}.
\end{array}
\right.
$$
So,
$$
\left\{
\begin{array}{ll}
(p-1-\frac{b}s)\frac{N-2s}{2N}\leq\frac{1-\delta}2-\frac{b}{2s}=\frac{p-1-\frac{b}s}{r_1}\leq\frac{p-1-\frac{b}s}2,\\
p\geq1+\frac{b}s,\\
0\leq\frac{N(1-\delta)}{1+N(1-\delta)}\leq s.
\end{array}
\right.
$$
Then,
$$
\left\{
\begin{array}{ll}
2\Big(\frac{1-\delta}2-\frac{b}{2s}\Big)\leq p-1-\frac{b}s\leq\frac{2N}{N-2s}\Big(\frac{1-\delta}2-\frac{b}{2s}\Big) ;\\
p\geq1+\frac{b}s;\\
0\leq\frac{N(1-\delta)}{1+N(1-\delta)}\leq s.
\end{array}
\right.
$$
Therefore,
$$
\left\{
\begin{array}{ll}
\max\{1+\frac{b}s,2-\delta\}\leq p\leq 1+\frac{b}s+\frac{N}{N-2s}\Big(1-\delta-\frac{b}{s}\Big);\\
\frac{N(1-\delta)}{1+N(1-\delta)}\leq s.
\end{array}
\right.
$$
The proof is achieved by taking $\delta\to0$.

\section{A remark on the defocusing case}\label{appendix2}
Although the main goal of this paper is concerned with the focusing regime, we want to draw the reader’s attention to the fact that the defocusing case can be studied in a more simpler way. For the sake of completeness, we give the statement and the proof in this appendix.

Consider the defocusing fractional inhomogeneous Schr\"odinger equation
\begin{align} \label{S'}
	\left\{
	\begin{array}{ccl}
	i\partial_t u -(-\Delta)^s u &=&  |x|^{-b}|u|^{p-1}u, \quad (t,x) \in \R \times \R^N,  \\
	u(0,x)  &= & u_0(x).
	\end{array}
	\right.
	\end{align}
One proves the following scattering result.
\begin{prop}\label{defoc}
 Assume that $ N\geq 3$, $s\in(\frac N{2N-1},1)$, $0<b<2s$ and $p_*<p<p^*$. Then, for $u_0\in H_{rad}^{s}$, the maximal energy solution of \eqref{S'} is global and scatters in $H^s$.
\end{prop}
\begin{proof}
The global existence follows by the conservation of the energy and the energy sub-critical regime. Let the Morawetz action
$$M_{f}[u(t)]:=2\Im\int_{\R^N}\bar {u}\nabla{f}\cdot\nabla u\,dx := 2\Im\int_{\R^N}\bar u\partial_k{f}\partial_k u\,dx. $$
Taking into account the calculus performed in the proof of Lemma \ref{Evo-M-R}, we get
\begin{eqnarray*}
\partial_t\Big(M_f[u(t)]\Big)
&=&\int_0^\infty m^s\int_{\R^N}\Big(4\partial_k\bar u_m\partial^2_{kl}{f}\partial_lu_m-\Delta^2{f}|u_m|^2\,\Big)dx\,dm\\
&+&\frac{2(p-1)}{1+p}\int_{\R^N}\Delta{f}|x|^{-b}|u|^{p+1}\,dx+\frac{4b}{1+p}\int_{\R^N}\frac{x\cdot\nabla{f}}{|x|^2}|x|^{-b}|u|^{1+p}\,dx.
\end{eqnarray*}
One picks $f(x):=|x|$. Then,
$$\nabla f(x)=\frac{x}{|x|},\quad \Delta f=\frac{N-1}{|x|},$$
and
$$-\Delta^2f=\left\{
\begin{array}{ll}
4\pi(N-1)\delta_0\quad\mbox{if}\quad N=3,\\
\frac{(N-1)(N-3)}{|\cdot|^3}\quad\mbox{if}\quad  N\geq4.
\end{array}
\right.$$
Clearly we have
$$\partial_l\partial_k f\Re(\partial_ku_m\partial_l\bar{u_m})\geq0,$$
where $u_m$ is given by \eqref{u-m}.
Indeed, denoting $r:=|x|$, one computes
\begin{align*}
\partial_l\partial_kf(x)&=\frac{1}{r}(\delta_{lk}-\frac{x_lx_k}{r^2});\\
\partial_l\partial_kf\Re(\partial_ku_m\partial_l\bar u_m)&=\frac{1}{r}\left(|\nabla u_m|^2-\frac{|x\cdot \nabla u_m|^2}{r^2}\right)\geq 0.
\end{align*}
Hence,
\begin{eqnarray*}
\partial_t\Big(M_f[u(t)]\Big)
&\geq&\frac{2(p-1)}{1+p}\int_{\R^N}\Delta{f}|x|^{-b}|u|^{p+1}\,dx+\frac{4b}{1+p}\int_{\R^N}\frac{x\cdot\nabla{f}}{|x|^2}|x|^{-b}|u|^{1+p}\,dx\\
&\geq&\frac{2(p-1)(N-1)}{1+p}\int_{\R^N}|x|^{-b-1}|u|^{p+1}\,dx+\frac{4b}{1+p}\int_{\R^N}|x|^{-b}|u|^{1+p}\,dx.
\end{eqnarray*}
Since by \cite[Lemma A.1]{bhl}, one has $|M_f[u(t)]|\leq C(\|\bD^{\frac12}u(t)\|^2+\|u(t)\|\|\bD^{\frac12}u(t)\|)$, one gets
$$\int_{\R^N}|x|^{-b-1}|u(t,x)|^{p+1}\,dx+\int_{\R^N}|x|^{-b}|u(t,x)|^{1+p}\,dx\lesssim 1.$$
Therefore, by the fractional radial Sobolev inequality \eqref{F-Sob-I}, we get
\begin{eqnarray*}
1
&\gtrsim&\int_{\R\times\R^N}|x|^{-b-1}|u(t,x)|^{p+1}\,dx\,dt+\int_{\R^N}|x|^{-b}|u(t,x)|^{1+p}\,dx\,dt\\
&\gtrsim&\int_{\R\times\R^N}|u(t,x)|^{1+\frac{2(1+b)}{N-2s}+p}\,dx\,dt+\int_{\R^N\times\R}|u(t,x)|^{1+\frac{2b}{N-2s}+p}\,dx\,dt.
\end{eqnarray*}
Hence,
$$u\in L^{1+\frac{2(1+b)}{N-2s}+p}\Big(\R, L^{1+\frac{2(1+b)}{N-2s}+p}(\R^N)\Big)\cap L^{1+\frac{2b}{N-2s}+p}\Big(\R,L^{1+\frac{2b}{N-2s}+p}(\R^N)\Big).$$
Now, since $0<s_c<1$ implies that $0<1+\frac{2b}{N-2s}+p<s$, one has $(1+\frac{2b}{N-2s}+p,1+\frac{2b}{N-2s}+p)\in\Gamma_\gamma$ for some $0<\gamma<s$. The scattering follows by standard arguments via Stichartz estimates.
\end{proof}

\end{appendix}




\begin{thebibliography}{99}

\bibitem{AT} {L. Aloui and S. Tayachi}, {\em Local well-posedness for the inhomogeneous nonlinear Schr\"odinger equation}, Discrete Cont. Dyn. Syst., 41 (2021), 5409--5437.

\bibitem{Adams}{ R. Adams}, {\em Sobolev spaces}, Academic. New York, (1975).

\bibitem{BCD}{H. {Bahouri}, J.-Y. {Chemin} and R. {Danchin}}, {{\em Fourier analysis and nonlinear partial differential equations}}, {{Grundlehren Math. Wiss.}}, Volume {343}, Pages {xvi + 523}, {2011}.

\bibitem{BL}{R. Bai and B. Li}, {\em Blow-up for the inhomogeneous nonlinear Schr\"odinger equation}, \textcolor{blue}{https://arxiv.org/abs/2103.13214}, 2021.

\bibitem{BDM-JHDE2018}{A. Bensouilah, D. Draouil and M. Majdoub}, {\em Energy critical {Schr{\"o}dinger} equation with weighted exponential nonlinearity: local and global well-posedness},
{J. Hyperbolic Differ. Equ.}, { 15} (2018), {599--621}.

\bibitem{BDM-CPAA2019} {A. Bensouilah, V. D. Dinh and M. Majdoub}, {\em Scattering in a weighted $L^2$ space for a 2-D Schr\"odinger equation with inhomogeneous exponential nonlinearity,} {Commun. Pure Appl. Anal.}, { 18} (2019), {2735--2755}.

\bibitem{BJ76} {J. Bergh and J. L\"ofstr\"om}, {\em Interpolation spaces}, Springer, New York, 1976.


\bibitem{bhl}{ T. Boulenger, D. Himmelsbach and E. Lenzmann}, {\em Blowup for fractional NLS}, J. of Funct. Anal., 271 (2016), 2569--2603.

\bibitem{Campos} {L. Campos}, {\em Scattering of radial solutions to the inhomogeneous nonlinear Schr\"odinger equation}, Nonlinear Anal., 202 (2021), 112118.

\bibitem{cc}{ L. Campos and M. Cardoso}, {\em A Virial-Morawetz approach to scattering for the non-radial inhomogeneous NLS}, Proc. Amer. Math. Soc., 150 (2022), 2007--2021.

\bibitem{CFGM} {M. Cardoso, L. G. Farah, C. M. Guzm\'an and J. Murphy}, {\em Scattering below the ground state for the intercritical non-radial inhomogeneous NLS}, Nonlinear Anal. Real World Appl. 68 (2022), 103687.

\bibitem{Caz-CLNM} {T. Cazenave}, {\em Semilinear Schr\"odinger equations}, Courant Lecture Notes in Mathematics, {10}. New York University, Courant Institute of Mathematical Sciences, AMS, 2003.

\bibitem{CG-DCDS-B} {J. Chen and B. Guo}, {\em Sharp global existence and blowing up results for inhomogeneous
		Schr\"odinger equations}, Discrete Contin. Dynam. Systems Series B, 8 (2007), 357--367.
		
\bibitem{Chen} {J. Chen}, {\em On a class of nonlinear inhomogeneous Schr\"odinger equation}, J. Appl. Math. Comput., 32 (2010), 237--253.
		
 \bibitem{Chen-CMJ} {J. Chen}, {\em On the inhomogeneous nonlinear Schr\"odinger equation with harmonic potential and unbounded coefficient}, Czech. Math. J., 60 (2010), 715--736.
		
 \bibitem{CG-AM} {J. Chen and B. Guo}, {\em Sharp constant of an improved Gagliardo-Nirenberg inequality and its application}, Annali di Matematica, 190 (2011), 341--354.

\bibitem{cl}{ Y. Cho and S. Lee}, {\em Strichartz estimates in spherical coordinates}, Indiana Univ. Math. J., 62 (2013), 991--1020. 

\bibitem{co}{ Y. Cho and T. Ozawa}, {\em Sobolev inequalities with symmetry}, Commun. Contemp. Math., 11 (2009), 355--365.

\bibitem{cox}{ Y. Cho, T. Ozawa and S. Xia}, {\em Remarks on some dispersive estimates}, Commun. Pure Appl. Anal., 10 (2011), 1121--1128.

\bibitem{cw}{ M. Christ and M. Weinstein}, {\em Dispersion of small amplitude solutions of the generalized Korteweg-de Vries equation}, J. Funct. Anal. 100 (1991), 87--109.

\bibitem{CIMM} J. Colliander, S. Ibrahim, M. Majdoub and N. Masmoudi, {\em Energy critical NLS in two space dimensions}, {J. Hyperbolic Differ. Equ.}, {6} (2009), 549--575.

\bibitem{BCS} {A. Cotsiolis and K.-N. Tavoularis}, {\em Best constants for {Sobolev} inequalities for higher order fractional derivatives}, {J. Math. Anal. Appl.}, { 295} (2004), {225--236}.

 
 \bibitem{Dinh-NA} V. D. Dinh, {\em Blowup of $H^1$ solutions for a class of the focusing inhomogeneous nonlinear Schr\"odinger equation}, Nonlinear Anal., 174 (2018), 169--188.
		
\bibitem{vdd}{ V. D. Dinh}, {\em Well-posedness of nonlinear fractional Schr\"odinger and wave equations in Sobolev spaces}, arXiv:1609.06181 [math.AP]
		
\bibitem{Dinh} {V. D. Dinh}, {\em Scattering theory in weighted $L^2$ space for a class of the defocusing inhomogeneous nonlinear Schr\"odinger equation}, Adv. Pure Appl. Math., 12 (2021), 38--72.
		
\bibitem{Dinh-2D} {V. D. Dinh}, {\em Energy scattering for a class of inhomogeneous nonlinear Schr\"odinger equations in two dimensions}, J. Hyper. Diff. Equ., 18 (2021), 1--28.
		
		
\bibitem{DK-SIAM} {V. D. Dinh and S. Keraani}, {\em Long time dynamics of non-radial solutions to inhomogeneous nonlinear Schr\"odinger equations}, SIAM J. Math. Anal., 54 (2021), 4765--4811.
		
\bibitem{dd}{ V. D. Dinh}, {\em A unified approach for energy scattering for focusing nonlinear schr\"odinger equations}, Discr. Cont. Dyn. Syst., 40 (2020), 6441-6471.

\bibitem{DKM}{V. D. {Dinh}, S. {Keraani} and M. {Majdoub}}, {{\em Long time dynamics for the focusing nonlinear Schr\"odinger equation with exponential nonlinearities}},
 {{Dyn. Partial Differ. Equ.}}, { 17} (2020), {329--360}.

\bibitem{DMS}{V. D. {Dinh}, M. {Majdoub} and T. Saanouni}, {\em Long time dynamics and blow-up for the focusing inhomogeneous nonlinear Schr\"odinger equation with spatial growing nonlinearity}, \textcolor{blue}{https://arxiv.org/abs/2107.01479}, 2021.

\bibitem{DM2017}{B. Dodson and J. Murphy}, {\em A new proof of scattering below the ground state for the 3D radial focusing cubic {NLS}}, {Proc. Am. Math. Soc.}, {145} (2017), {4859--4867}.

\bibitem{dhr}{T. Duyckaerts, J. Holmer and S. Roudenko}, {\em Scattering for the non-radial 3D cubic nonlinear
Schr\"odinger equation}, Math. Res. Lett., 15 (2008), 1233-1250.

 \bibitem{Farah} {L. G. Farah}, {\em Global well-posedness and blow-up on the energy space for the inhomogeneous nonlinear Schr\"odinger equation}, J. Evol. Equ., 16 (2016), 193--208.
		
\bibitem{FG-JDE} {L. G. Farah and C. M. Guzm\'an}, {\em Scattering for the radial 3D cubic focusing inhomogeneous nonlinear Schr\"odinger equation}, J. Diff. Equa., 262 (2017), 4175--4231.

	\bibitem{FG-BBMS} {L. G. Farah and C. M. Guzm\'an}, {\em Scattering for the radial focusing inhomogeneous NLS equation in higher dimensions}, Bull. Braz. Math. Soc. (N.S.), 51 (2020), 449--512.
		

 \bibitem{Fib-Book}{G. Fibich}, {\em The Nonlinear Schr\"odinger Equation, Singular Solutions and Optical Collapse},  Applied Mathematical Sciences, Book series (AMS, volume 192), 2015.
 
 
 \bibitem{FJL07}{ J. Fröhlich, B. Lars G. Jonsson and E. Lenzmann}, {\em Boson stars as solitary waves}, Comm. Math. Phys., 274 ( 2007), 1-30.

\bibitem{GS} {F. Genoud and C. A. Stuart}, {\em Schr\"odinger equations with a spatially decaying nonlinearity: existence and stability of standing waves}, Discrete Contin. Dyn. Syst., 21 (2008), 137--18.

 \bibitem{ghs}{R. Ghanmi, H. Hezzi and T. Saanouni}, {\em A note on inhomogeneous coupled Schr\"odinger equations}, Ann. Henri Poincar\'e, {21} (2020), 2775-–2814.

 \bibitem{gh1}{B. Guo and Z. Huo}, {\em Global well-posedness for the fractional nonlinear Schr\"odinger equation}, Commun. Partial Differ. Equ., 36 (2011), 247--255.

 \bibitem{gh2}{B. Guo and Z. Huo}, {\em Well-posedness for the nonlinear fractional Schr\"odinger equation and inviscid limit behavior of solution for the fractional Ginzburg-Landau equation}, Frac. Calc. Appl. Anal., 16 (2013), 226--242.

\bibitem{gw1}{ Z. Guo and Y. Wang}, {\em Improved Strichartz estimates for a class of dispersive equations in the radial case and
their applications to nonlinear Schr\"odinger and wave equations}, J. Anal. Math., 124 (2014), 1--38. 

\bibitem{Guo2018}{Z. {Guo}, Y. {Sire}, Y. {Wang} and L. {Zhao}}, {{\em On the energy-critical fractional Schr\"odinger equation in the radial case}}, {{Dyn. Partial Differ. Equ.}}, {15} (2018), {265--282}.

\bibitem{Guzman} {C. M. Guzm\'an}, {\em On well posedness for the inhomogeneous nonlinear Schr\"odinger equation}, Nonlinear Anal. Real World Appl., 37 (2017), 249--286.

\bibitem{yhys}{Y. {Hong} and Y. {Sire}}, {\em On fractional Schr\"odinger equations in sobolev spaces}, {Comm. Pur. Appl.Anal.}, { 14} (2015), {2265--2282}.


\bibitem{Nonlinearity} {S. Ibrahim, M. Majdoub, N. Masmoudi and K. Nakanishi}, {\em Scattering for the two-dimensional {NLS} with exponential nonlinearity}, {Nonlinearity}, {25} (2012), {1843--1849}.

\bibitem{Merle}{ C. E. Kenig and F. Merle}, {\em Global well-posedness, scattering and blow up for the energy critical, focusing, non-linear Schr\"odinger equation in the radial case}, Invent. Math., 166 (2006), 645--675.
\bibitem{kpv}{ C. E. Kenig, G. Ponce and L. Vega}, {\em Well-posedness and scattering results for the generalized Korteweg-de Vries equation via the contraction principle}, Comm. Pure Appl. Math., 46 (1993), 527--620.

\bibitem{KLS} {J. Kim, Y. Lee and I. Seo}, {\em On well-posedness for the inhomogeneous nonlinear Schr\"odinger equation in the critical case}, J. Diff. Equ., 280 (2021), 179--202.

\bibitem{KLS13}{ K. Kirkpatrick, E. Lenzmann and G. Staffilani}, {\em  On the continuum limit for discrete NLS with long-range lattice interactions}, Comm. Math. Phy., 317 (2013), 563--591.

\bibitem{MMZ} {C. Miao, J. Murphy and J. Zheng}, {\em Scattering for the non-radial inhomogeneous NLS}, Mathematical Research Letters, 28 (2021), 1481--1504.
		
\bibitem{Murphy} {J. Murphy}, {\em A simple proof of scattering for the intercritical inhomogeneous NLS}, Proc. Amer. Math. Soc., 150 (2022), 1177--1186.
  
\bibitem{pz}{ C. Peng and D. Zhao}, {\em Global existence and blowup on the energy space for the inhomogeneous fractional nonlinear Schr\"odinger equation}, Discr. Cont. Dyn. Syst. Series B, 24 (2019), 3335--3356.

\bibitem{ts} {T. Saanouni}, {\em Remark on the inhomogeneous fractional nonlinear Schr\"odinger equations}, J. Math. Phys., 57 (2016), 081503.

\bibitem{st4}{ T. Saanouni}, {\em Energy scattering for radial focusing inhomogeneous bi-harmonic Schr\"odinger equations}, Calc. Var., (2021) 60:113.

\bibitem{syz}{ C. Sun, H. Wang, X. Yao and J. Zheng}, {\em Scattering below ground state of focusing fractional nonlinear schr\"odinger equation with radial data}, Discr. Cont. Dyn. Syst., 38 (2018), 2207--2228.

\bibitem{SS99} {C. Sulem and P. Sulem}, {\em Nonlinear Schr\"odinger Equation: Self-Focusing and Wave Collapse}, Springer, 1999.

\bibitem{Strauss} W. A. Strauss, {\em Decay and asymptotics for {{\(\square u = F(u)\)}}}, {J. Funct. Anal.}, {2} (1968), 409--457.

\bibitem{Tao06} {T. Tao}, {\em Nonlinear dispersive equations. Local and global analysis}, CBMS Regional Conference Series in Mathematics, 106, American Mathematical Society, 2006.

\bibitem{tao1} {T. Tao}, {\em On the asymptotic behavior of large radial data for a focusing non-linear
Schr\"odinger equation}. Dyn. Partial Differ. Equ., 1 (2004), 1--48.

\bibitem{Val09}{ E. Valdinoci}, {\em From the long jump random walk to the fractional Laplacian}, Bol. Soc. Esp. Mat. Apl., 49 (2009), 33--44.

\end{thebibliography}
\end{document}